\documentclass[letterpaper,10pt]{amsart}
\usepackage[bb=fourier]{mathalfa}
\usepackage{amsmath}
\usepackage{amssymb}
\usepackage{mathtools}      % improves amsmath
\usepackage{stmaryrd} % the commands \llbracket, \rrbracket replace \ldbrack, \rdbrack from package mathabx
\usepackage{esint} % extended set of integrals for Computer Modern
\usepackage{enumitem}
\usepackage{amsthm}
\usepackage{thmtools}
\usepackage{tikz}
\usepackage{tikz-cd}
\usepackage[colorlinks=true,backref=page,bookmarksopen=true]{hyperref} % Links in the pdf. Backref option: each bibliographical entry denotes where it was cited
\usepackage[capitalise]{cleveref} % Must be loaded after hyperref. It's advisable to compile the file with LUALATEX instead of PDFLATEX to avoid some annoying warning messages.
\usepackage[font=small]{caption}

\usepackage{iftex}
\ifPDFTeX
   \usepackage[utf8]{inputenc}
   \usepackage[T1]{fontenc}
   \usepackage{lmodern}
\else
   \ifXeTeX
     \usepackage{fontspec}
   \else
     \usepackage{luatextra}
   \fi
   \defaultfontfeatures{Ligatures=TeX}
\fi

\usepackage{xcolor}
\hypersetup{
    colorlinks,
    linkcolor={red!50!black},
    citecolor={blue!50!black},
    urlcolor={blue!80!black}
}

\renewcommand*{\backref}[1]{}
\renewcommand*{\backrefalt}[4]{\quad \tiny 
  \ifcase #1 (\textbf{NOT CITED.})%
  \or    (Cited on page~#2.)%
  \else   (Cited on pages~#2.)%
  \fi}

\DeclareFontFamily{U} {MnSymbolA}{}
\DeclareFontShape{U}{MnSymbolA}{m}{n}{
   <-6> MnSymbolA5
   <6-7> MnSymbolA6
   <7-8> MnSymbolA7
   <8-9> MnSymbolA8
   <9-10> MnSymbolA9
   <10-12> MnSymbolA10
   <12-> MnSymbolA12}{}
\DeclareFontShape{U}{MnSymbolA}{b}{n}{
   <-6> MnSymbolA-Bold5
   <6-7> MnSymbolA-Bold6
   <7-8> MnSymbolA-Bold7
   <8-9> MnSymbolA-Bold8
   <9-10> MnSymbolA-Bold9
   <10-12> MnSymbolA-Bold10
   <12-> MnSymbolA-Bold12}{}
\DeclareSymbolFont{MnSyA} {U} {MnSymbolA}{m}{n}
 \DeclareFontFamily{U} {MnSymbolC}{}
\DeclareFontShape{U}{MnSymbolC}{m}{n}{
  <-6> MnSymbolC5
  <6-7> MnSymbolC6
  <7-8> MnSymbolC7
  <8-9> MnSymbolC8
  <9-10> MnSymbolC9
  <10-12> MnSymbolC10
  <12-> MnSymbolC12}{}
\DeclareFontShape{U}{MnSymbolC}{b}{n}{
  <-6> MnSymbolC-Bold5
  <6-7> MnSymbolC-Bold6
  <7-8> MnSymbolC-Bold7
  <8-9> MnSymbolC-Bold8
  <9-10> MnSymbolC-Bold9
  <10-12> MnSymbolC-Bold10
  <12-> MnSymbolC-Bold12}{}
\DeclareSymbolFont{MnSyC} {U} {MnSymbolC}{m}{n}

\DeclareMathSymbol{\top}{\mathord}{MnSyA}{219} % smaller symbol for transpose
\DeclareMathSymbol{\plus}{\mathord}{MnSyC}{20} % a smaller plus sign

\declaretheorem[numberwithin=section]{theorem} 
\declaretheorem[sibling=theorem]{proposition} 
\declaretheorem[sibling=theorem]{lemma}
\declaretheorem[sibling=theorem]{corollary}

\declaretheorem[sibling=theorem]{question}

\declaretheorem[sibling=theorem, style=definition]{example}
\declaretheorem[sibling=theorem, style=definition]{definition}
\declaretheorem[sibling=theorem, style=remark]{remark}

\hypersetup{bookmarksdepth = 3} % Depth of sections/subs... to have bookmark links in the pdf;
\numberwithin{equation}{section}     % Makes labeled equations easier to find.

\setlist[enumerate,1]{label={\upshape(\alph*)},ref=\alph*}
\newcommand{\N}{\mathbb{N}}   \newcommand{\R}{\mathbb{R}}
\ifluatex
	
\else
	
\fi
\renewcommand{\P}{\mathbb{P}}
\newcommand{\cC}{\mathcal{C}}
\newcommand{\cF}{\mathcal{F}}

\newcommand{\cP}{\mathcal{P}}
\newcommand{\cS}{\mathcal{S}}\newcommand{\cU}{\mathcal{U}}

\newcommand{\st}{\;\mathord{;}\;}
\newcommand*\circled[1]{\tikz[baseline=(char.base)]{
    \node[shape=circle,draw,inner sep=1pt] (char) {\footnotesize{#1}};}}

\DeclareMathOperator{\per}{per}
\newcommand{\sym}{\mathsf{sym}}
\newcommand{\hsm}{\mathsf{hsm}}
\newcommand{\pmean}{\mathsf{pm}}
\newcommand{\smean}{\mathsf{sm}}
\newcommand{\gmean}{\mathsf{gm}}
\newcommand{\PHS}{{\cP_\mathrm{HS}(\R_{\plus})}}
\newcommand{\Pa}{\cP_1(\R_\plus)}
\newcommand{\Pg}{\cP_\mathrm{g}(\R_{\plus\plus})}

\renewcommand{\setminus}{\smallsetminus}

\renewcommand{\epsilon}{\varepsilon}

\newcommand{\doi}[1]{\href{http://dx.doi.org/#1}{\tt DOI}} %{DOI:{#1}}}
\newcommand{\directlink}[1]{\href{#1}{\tt URL}}
\newcommand{\MRev}[1]{\href{https://mathscinet.ams.org/mathscinet-getitem?mr=#1}{\tt MR}} 
\newcommand{\zb}[1]{\href{https://zbmath.org/?q=an:#1}{\tt ZB}} %{\tt zbMATH}}

\begin{document}

\title{The Halász--Székely Barycenter}
\date{\today}

\author[J.~Bochi]{Jairo Bochi}
\address{Facultad de Matem\'aticas, Pontificia Universidad Cat\'olica de Chile}
\email{\href{mailto:jairo.bochi@mat.uc.cl}{jairo.bochi@mat.uc.cl}}

\author[G.~Iommi]{Godofredo Iommi}
\address{Facultad de Matem\'aticas, Pontificia Universidad Cat\'olica de Chile}
\email{\href{mailto:giommi@mat.uc.cl}{giommi@mat.uc.cl}}

\author[M.~Ponce]{Mario Ponce}
\address{Facultad de Matem\'aticas, Pontificia Universidad Cat\'olica de Chile}
\email{\href{mailto:mponcea@mat.uc.cl}{mponcea@mat.uc.cl}} 


\begin{thanks}
{The authors were partially supported by CONICYT PIA ACT172001.
J.B.\ was partially supported  by Proyecto Fondecyt 1180371. 
G.I.\ was partially supported  by Proyecto Fondecyt 1190194.
M.P. \ was partially supported  by Proyecto Fondecyt 1180922.}
\end{thanks}

% https://mathscinet.ams.org/mathscinet/msc/msc2020.html
% 15-XX			Linear and multilinear algebra; matrix theory
%   15Axx		Basic linear algebra
%       15A15  	Determinants, permanents, traces, other special matrix functions [See also 19B10, 19B14]
% 26-XX			Real functions [See also 54C30]
%   26Exx		Miscellaneous topics in real functions [See also 58Cxx]
%       26E60  	Means [See also 47A64]
%   26Dxx		Inequalities in real analysis {For maximal function inequalities, see 42B25; for functional inequalities, see 39B72; for probabilistic inequalities, see 60E15}
%       26D15  	Inequalities for sums, series and integrals
% 37-XX			Dynamical systems and ergodic theory [See also 26A18, 28Dxx, 34Cxx, 34Dxx, 35Bxx, 46Lxx, 58Jxx, 70-XX]
%   37Axx		Ergodic theory [See also 28Dxx]
%       37A30  	Ergodic theorems, spectral theory, Markov operators {For operator ergodic theory, see mainly 47A35}
% 60-XX			Probability theory and stochastic processes {For additional applications, see 05Cxx, 11Kxx, 34-XX, 35-XX, 62-XX, 90-XX, 76-XX, 81-XX, 82-XX, 91-XX, 92-XX, 93-XX, 94-XX}
%   60Bxx		Probability theory on algebraic and topological structures
%       60B20  	Random matrices (probabilistic aspects) {For algebraic aspects, see 15B52}
%	60Exx		Distribution theory [See also 62Exx, 62Hxx]
%       60E15  	Inequalities; stochastic orderings
%   60Fxx		Limit theorems in probability theory [See also 28Dxx, 60B12]
%       60F15  	Strong limit theorems

\maketitle

\begin{abstract}
We introduce a notion of barycenter of a probability measure related to the symmetric mean of a collection of nonnegative real numbers. Our definition is inspired by the work of Hal\'asz and Sz\'ekely, who in 1976 proved a law of large numbers for symmetric means. We study analytic properties of this Hal\'asz--Sz\'ekely barycenter. We establish fundamental inequalities that relate the symmetric mean of a list of nonnegative real numbers with the barycenter of the measure uniformly supported on these points. As consequence, we go on to establish an ergodic theorem stating that the symmetric means of a sequence of dynamical observations converges to  the Hal\'asz--Sz\'ekely barycenter of the corresponding distribution.
\end{abstract}

%%%%%%%%%%%%%%%%%%%%%%%%%%%%%%%%%%%%%%%%%%%%%%%
\section{Introduction}
%%%%%%%%%%%%%%%%%%%%%%%%%%%%%%%%%%%%%%%%%%%%%%%

\emph{Means} have fascinated man for a long time. % The Babylonians knew
Ancient Greeks knew the arithmetic, geometric, and harmonic means of two positive numbers (which they may have learned from the Babylonians); they also studied other types of means that can be defined using proportions: see \cite[pp.~85--89]{Heath}. Newton and Maclaurin encountered the symmetric means (more about them later). Huygens introduced the notion of expected value and Jacob Bernoulli proved the first rigorous version of the law of large numbers: see \cite[pp.~51, 73]{Maistrov}.
Gauss and Lagrange exploited the connection between the arithmetico-geometric mean and elliptic functions: see \cite{Borweins}. 
Kolmogorov and other authors considered means from an axiomatic point of view and determined when a mean is arithmetic under a change of coordinates (i.e.\ \emph{quasiarithmetic}): see \cite[p.~157--163]{HLP}, \cite[Chapter~17]{AD}.
Means and inequalities between them are the main theme of the classical book \cite{HLP} by Hardy, Littlewood, and Pólya, and the book \cite{Bullen} by Bullen is a comprehensive account of the subject. 
Going beyond the real line, there are notions of averaging that relate to the geometric structure of the ambient space: see e.g.\ \cite{Stone,EM,Navas,KLL}.

In this paper, we are interested in one of the most classical types of means: the elementary symmetric polynomials means, or symmetric means for short. Let us recall their definition. Given integers $n \ge k \ge 1$,
the \emph{$k$-th symmetric mean} of a list of nonnegative numbers $x_1,\dots,x_n$  is:
\begin{equation}\label{e.def_sym_mean}
\sym_k(x_1,\dots,x_n) \coloneqq \left(\frac{E^{(n)}_k(x_1,\dots,x_n)}{\binom{n}{k}}\right)^{\frac{1}{k}} \, ,
\end{equation}
where $E^{(n)}_k(x_1,\dots,x_n) \coloneqq \sum_{i_1<\cdots<i_k} x_{i_1} \cdots x_{i_k}$
is the elementary symmetric polynomial of degree $k$ in $n$ variables.
Note that the extremal cases $k=1$ and $k=n$ correspond to arithmetic and the geometric means, respectively. The symmetric means are non-increasing as functions of $k$: this is Maclaurin's inequality: see \cite[p.~52]{HLP} or \cite[p.~327]{Bullen}.
For much more information on symmetric means and their relatives, see \cite[Chapter~V]{Bullen}.

Let us now turn to Probability Theory.
A law of large numbers in terms of symmetric means was obtained by Halász and Székely \cite{HS76}, confirming a conjecture of Székely \cite{Szekely_first}.
Let $X_1$, $X_2$, \dots be a sequence of nonnegative independent identically distributed random variables, and from them we form another sequence of random variables:
\begin{equation}
S_n \coloneqq \sym_k (X_1,\dots,X_n) \, ,
\end{equation}
The case of $k=1$ corresponds to the setting of the usual law of large numbers.
The case of constant $k>1$ is not significantly different from the classical setting. 
Things become more interesting if $k$ is allowed to depend on $n$, and it turns out to be advantageous to assume that $k/n$ converges to some number $c \in [0,1]$.
In this case, Halász and Székely \cite{HS76} have proved that if 
$X = X_1$ 
is strictly positive and satisfies some integrability conditions, then $S_n$ converges almost surely to a non-random constant.
% \margin{Added ``non-random''.}
Furthermore, they gave a formula for this limit, which we call the \emph{Halász--Székely mean} with parameter $c$ of the random variable $X$.
Halász and Székely theorem was extended to the nonnegative situation by van~Es \cite{vanEs} (with appropriate extra hypotheses). 
The simplest example consists of a random variable $X$ that takes two nonnegative values $x$ and $y$, each with probability $1/2$, and $c=1/2$; in this case the Halász--Székely mean is $\left(\frac{\sqrt{x}+\sqrt{y}}{2}\right)^2$. But this example is misleadingly simple, and Halász--Székely means are in general unrelated to power means.

Fixed the parameter $c$, the Halász--Székely mean of a nonnegative random variable $X$ only depends on its  distribution, which we regard as a probability measure $\mu$ on the half-line $[0,+\infty)$. 
Now we shift our point of view and consider probability measures as the fundamental objects. Instead of speaking of the mean of a probability measure, we prefer the word \emph{barycenter}, reserving the word \emph{mean} for lists of numbers (with or without weights), functions, and random variables. 
This is more than a lexical change.
The space of probability measures has a great deal of structure: it is a convex space and it can be endowed with several topologies. 
So we arrive at the notion of \emph{Halász--Székely barycenter} (or \emph{HS barycenter}) of a probability measure $\mu$ with parameter $c$, which we denote $[\mu]_c$. This is the subject of this paper. It turns out that HS barycenters can be defined directly, without resort to symmetric means or laws of large numbers (see \cref{def.HS}).

Symmetric means are intrinsically discrete objects and do not make sense as barycenters. In \cite[Remark, p.~323]{Bullen}, Bullen briefly proposes a definition of a weighted symmetric mean, only to conclude that ``the properties of this weighted mean are not satisfactory'' and therefore not worthy of further consideration. 
On the other hand, given a finite list $\underline{x} = (x_1,\dots,x_n)$ of nonnegative numbers, we can compare the symmetric means of $\underline{x}$ with the HS barycenter of the associated probability measure $\mu \coloneqq (\delta_{x_1}+\dots+\delta_{x_n})/n$. It turns out that these quantities obey certain precise inequalities (see \cref{t.main}). In particular, we have:
\begin{equation}
\sym_k(\underline{x}) \ge [\mu]_{k/n} \, .
\end{equation}
Furthermore, if $\underline{x}^{(m)}$ denotes the $nm$-tuple obtained by concatenation of $m$ copies of $\underline{x}$, then
\begin{equation}
[\mu]_{k/n} = \lim_{m \to \infty} \sym_{km}\big( \underline{x}^{(m)} \big) \, ,
\end{equation}
and we have precise bounds for the relative error of this approximation, depending only on the parameters and not on the numbers $x_i$ themselves.

Being a natural limit of symmetric means, the HS barycenters deserve to be studied by their own right. One can even argue that they give the ``right'' notion of weighted symmetric means that Bullen was looking for. HS barycenters have rich theoretical properties. They are also cheap to compute, while computing symmetric means involves summing exponentially many terms.

Using our general inequalities and certain continuity properties of the HS barycenters, we are able to obtain in straightforward manner an ergodic theorem that extends the laws of large numbers of Halász--Székely \cite{HS76} and van~Es \cite{vanEs}.

A prominent feature of the symmetric mean \eqref{e.def_sym_mean} is that it vanishes whenever more than $n-k$  of the numbers $x_i$ vanish. Consequently, the HS barycenter $[\mu]_c$ of a probability measure $\mu$ on $[0,+\infty)$ vanishes when $\mu(\{0\}) > 1-c$. 
In other words, once the mass of leftmost point $0$ exceeds the critical value $1-c$, then it imposes itself on the whole distribution, and suddenly forces the mean to agree with it. 
%Perhaps similar phenomena is observed in some real-world situations.
Fortunately, in the subcritical regime, $\mu(\{0\}) < 1-c$, the HS barycenter turns out to be much better behaved. As it will be seen in \cref{s.properties}, in the critical case $\mu(\{0\}) = 1-c$ the HS barycenter can be either positive or zero, so the HS barycenter can actually vary discontinuously. 
Therefore our regularity results and the ergodic theorem must take this critical phenomenon into account. 
%\margin{Added a couple of sentences.}

This article is organized as follows.
In \cref{s.properties}, we define formally the HS barycenters and prove some of their basic properties. 
In \cref{s.inequalities}, we state and prove the fundamental inequalities relating HS barycenters to symmetric means.
In \cref{s.continuity}, we study the problem of continuity of the HS barycenters with respect to appropriate topologies on spaces of probability measures.
In \cref{s.ergodic}, we apply the results of the previous sections and derive a general ergodic theorem (law of large numbers) for symmetric and HS means. 
%, which extends the results by Halász and Székely \cite{HS76} and van~Es \cite{vanEs}.
In \cref{s.concavity}, we turn back to fundamentals and discuss concavity properties of the HS barycenters and means. 
Finally, in \cref{s.DHS} we introduce a different kind of barycenter which is a natural approximation of the HS barycenter, but has in a sense simpler theoretical properties.

% {\color{red}

% Logical dependence between sections:
% \begin{center}
% \begin{tikzcd}[row sep=small]
%     &\ref{s.inequalities}   \arrow[rd]  \\
% \ref{s.properties} \arrow[ru]\arrow[r]\arrow[rd]   &\ref{s.continuity} \arrow[r]   &\ref{s.ergodic}    \\
%     &\ref{s.concavity}      \arrow[r]   &\ref{s.DHS}
% \end{tikzcd}
% \end{center}
% }

%%%%%%%%%%%%%%%%%%%%%%%%%%%%%%%%%%%%%%%%%%%%%%%
\section{Presenting the HS barycenter}\label{s.properties}
%%%%%%%%%%%%%%%%%%%%%%%%%%%%%%%%%%%%%%%%%%%%%%%

%\margin{Changed the title of the section, to avoid the same title of a (new) subsection. Is ``Presenting'' ok or too informal?}

%\margin{Also rewrote the intro of the section}
Hardy, Littlewood, and Pólya's axiomatization of (quasiarithmetic) means \cite[\S~6.19]{HLP} is formulated in terms of distribution functions, using Stieltjes integrals. 
Since the first publication of their book in 1934, \emph{measures} became established as fundamental objects in mathematical analysis, probability theory, dynamical systems, etc. 
Spaces of measures have been investigated in depth (see e.g.\ the influential books \cite{Partha_book,Villani}).
The measure-theoretic point of view provides the convenient structure for the analytic study of means or, as we prefer to call them in this case, \emph{barycenters}. 
The simplest example of barycenter is of course the ``arithmetic barycenter'' of a probability measure $\mu$ on Euclidean space $\R^d$, defined (under the appropriate integrability condition) as $\int x   \, d \mu(x)$. 
%\margin{I think that ``barycenter'' has a more geometrical meaning, that's why I defined the arith.bar. on $\R^d$.}
Another example is the ``geometric barycenter'' of a probability measure $\mu$ on the half-line $(0,+\infty)$, defined as $\exp \left( \int  \log x \, d\mu(x) \right)$.
In this section, we introduce the Hallász--Székely barycenters and study some of their basic properties.

\subsection{Definitions and basic properties}
Throughout this paper we use the following notations:
\begin{equation}
\R_\plus \coloneqq [0,+\infty) \, , \quad \R_{\plus\plus} \coloneqq (0,+\infty) \, .
\end{equation}
We routinely work with the extended line $[-\infty,+\infty]$, endowed with the order topology.

\begin{definition}\label{def.kernel}
The  \emph{Hal\'asz--Sz\'ekely kernel}  (or \emph{HS kernel}) is the following function of three variables $x \in \R_\plus$, $y \in \R_{\plus\plus}$, and $c \in [0,1]$: 
\begin{equation}\label{e.def_kernel}
K(x,y,c) \coloneqq 
\begin{cases}
\log y + c^{-1} \log \left( cy^{-1}x + 1 - c \right) & \text{if } c>0, \\ 
\log y + y^{-1} x -1 & \text{if } c=0. 
\end{cases}
\end{equation}
\end{definition}

\begin{proposition} \label{p.kernel}
The HS kernel has the following properties (see also \cref{fig.K}):
\begin{enumerate}
\item\label{i.kernel_cont} The function $K \colon [0,+\infty) \times (0,+\infty) \times [0,1] \to [-\infty,+\infty)$ is continuous, attaining the value  $-\infty$ only at the points $(0,y,1)$.
\item\label{i.kernel_mono_x} $K(x,y,c)$ is increasing with respect to $x$.
\item\label{i.kernel_mono_c} $K(x,y,c)$ is decreasing with respect to $c$, and strictly decreasing when $x\neq y$.
\item\label{i.kernel_c_1} $K(x,y,1) = \log x$ is independent of $y$.
\item\label{i.kernel_reflex} $K(x,y,c) \ge \log x$, with equality if and only if $x=y$ or $c=1$.
\item For each $y>0$, the function $K(\mathord{\cdot},y,0)$ is affine, and its graph is the tangent line to $\log x$ at $x=y$. 
\item\label{i.kernel_homog} $K(\lambda x, \lambda y, c) = K(x,y,c) + \log \lambda$, for all $\lambda>0$.
\end{enumerate}
\end{proposition}

\begin{proof}
Most properties are immediate from \cref{def.kernel}.
To check monotonicity with respect to $c$, we compute the partial derivative when $c>0$:
\begin{equation}
K_c(x,y,c) =
\frac{1}{c^2} \left[ \frac{c(y^{-1}x-1)}{cy^{-1}x + 1 - c} + \log \left( \frac{1}{cy^{-1}x + 1 - c} \right) \right] \, ;
\end{equation}

since $\log t \le t-1$ (with equality only if $t=1$), we conclude that $K_c(x,y,c) \le 0$ (with equality only if $x=y$).
Since $K$ is continuous, we obtain property~\eqref{i.kernel_mono_c}.
Property \eqref{i.kernel_reflex} is a consequence of properties \eqref{i.kernel_mono_c} and \eqref{i.kernel_c_1}.
\end{proof}

% Maple code:
% K:=(x,y,c)->`if`(c>0,log(y)+c^(-1)*log(c/y*x+1-c),log(y)+ x/y - 1);
% display(
% plot([K(x,2,0),K(x,2,1/3),K(x,2,2/3),K(x,2,1)],
% x=0..4,K=-1.5..2,scaling=constrained,color="#78000E"),
% line([2,0],[2,ln(2)],linestyle=dash,thickness=0),
% line([2,ln(2)],[0,ln(2)],linestyle=dash,thickness=0),
% textplot([0, ln(2), "log 2"],'align'={'left'},'font'=["times","roman",10]));
\begin{figure}[htb]
\begin{center}
\includegraphics[width=.65\textwidth]{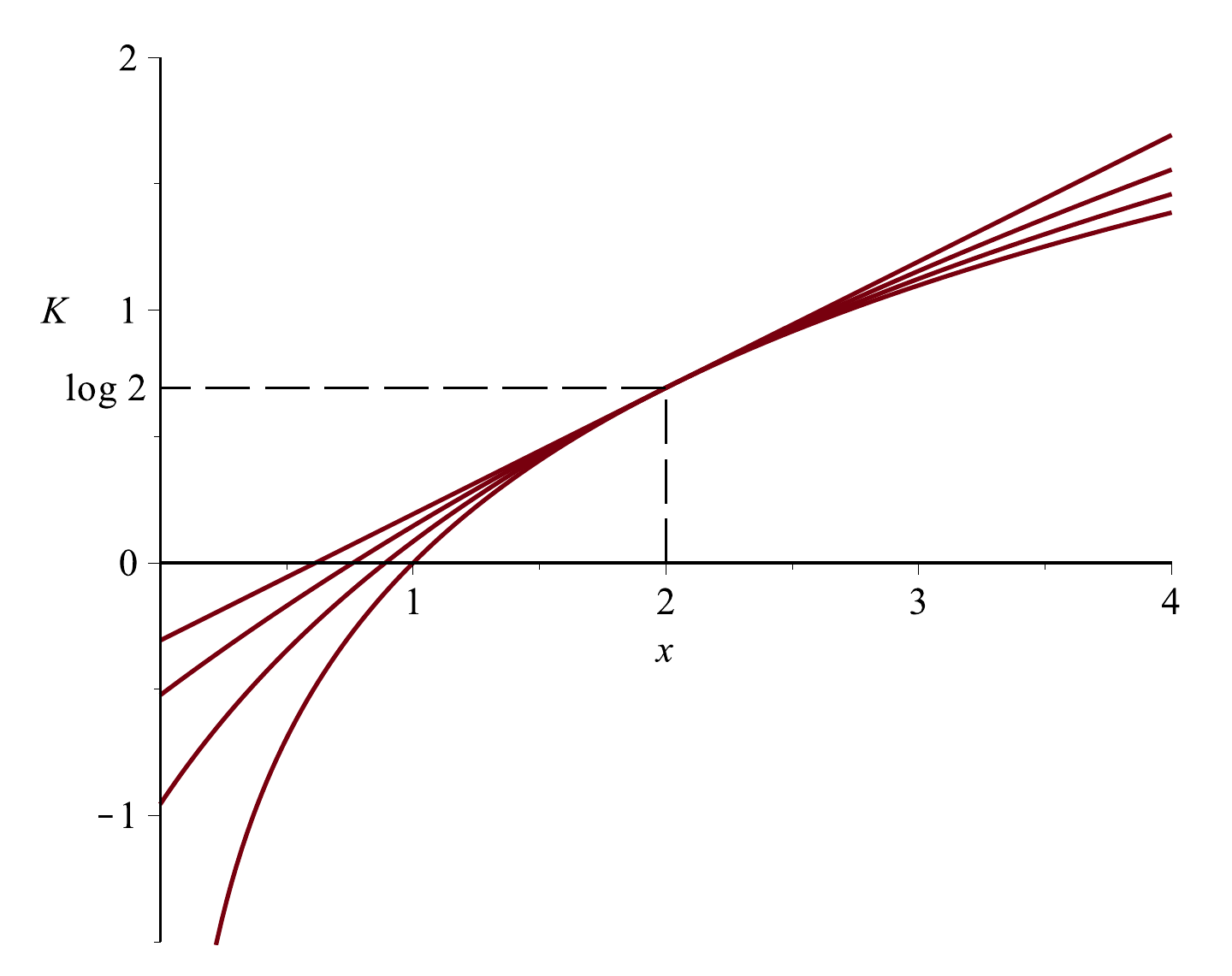}
\caption{Graphs of the functions $K(\mathord{\cdot},y,c)$ for $y=2$ and $c \in \{0,1/3,2/3,1\}$.}\label{fig.K}
\end{center}
\end{figure}

\medskip

Let $\cP(\R_\plus)$ denote the set of all Borel probability measures $\mu$ on $\R_\plus$.
The following is the central concept of this paper:

\begin{definition}\label{def.HS}
Let $c \in [0,1]$ and $\mu \in \cP(\R_\plus)$.
If $c=1$, then we require that the function $\log x$ is semi-integrable\footnote{A function $f$ is called \emph{semi-integrable} if the positive part $f^+ \coloneqq \max(f,0)$ is integrable or the negative part $f^- \coloneqq \max(-f,0)$ is integrable.} with respect to $\mu$. % that is, at least one of the functions $\log^+(x)$ and $\log^-(x)$ is integrable.
%\margin{Defined semi-integrability in a footnote.}
The \emph{Hal\'asz--Sz\'ekely barycenter} (or \emph{HS barycenter}) with parameter $c$ of the probability measure $\mu$ is:
\begin{equation}\label{e.def_HS}
[\mu]_c \coloneqq  \exp \inf_{y>0} \int K(x,y,c) \, d\mu(x) \, ,
\end{equation}
where $K$ is the HS kernel \eqref{e.def_kernel}.
\end{definition}

First of all, let us see that the definition is meaningful:
\begin{itemize}
\item If $c<1$, then for all $y>0$, the function $K( \mathord{\cdot}, y, c)$ is bounded from below by $K(0,y,c)>-\infty$, and therefore it has a well-defined integral (possibly $+\infty$); so $[\mu]_c$ is a well-defined element of the extended half-line $[0,+\infty]$.
\item If $c=1$, then by part \eqref{i.kernel_c_1} of \cref{p.kernel}, the defining formula \eqref{e.def_HS} becomes: 
\begin{equation}\label{e.HS1}
[\mu]_1 = \exp \int \log x \, d\mu(x)  \, .
\end{equation}
% (where, as usual, $\log 0 \coloneqq -\infty$).
The integral is a well-defined element of $[-\infty, +\infty]$, so $[\mu]_1$ is well-defined in $[0,+\infty]$. 
\end{itemize}

Formula \eqref{e.HS1} means that the HS barycenter with parameter $c=1$ is the geometric barycenter; let us see that $c=0$ corresponds to the standard arithmetic barycenter: 
%\margin{Removed quotes here (but kept them in the beggining of the section).}

\begin{proposition}\label{p.HS0}
For any $\mu \in \cP(\R_\plus)$, we have $[\mu]_0 = \int x \, d\mu(x)$.
\end{proposition}

\begin{proof}
Let $a \coloneqq \int x \, d\mu(x)$.
If $a = \infty$, then for every $y>0$, the non-constant affine function $K(\mathord{\cdot},y,0)$ has infinite integral, so definition \eqref{e.def_HS} gives $[\mu]_0 = \infty$.
On the other hand, if $a<\infty$, then $[\mu]_0$ is defined as $\exp \inf_{y>0} (\log y + a/y -1) = a$. %; this infimum is attained at $y=a$, yielding $[\mu]_0 = a$.
\end{proof}

Let $\PHS$ denote the subset formed by those $\mu \in \cP(\R_\plus)$ 
such that:
\begin{equation}\label{e.HS_integrability}
\int  \log (1 + x) \, d\mu(x) < \infty 
\end{equation}
or, equivalently, $\int  \log^+ x \, d\mu < \infty$ (we will sometimes write ``$d\mu$'' instead of ``$d\mu(x)$''). 

\begin{proposition}\label{p.HS_infinity}
Let $c \in (0,1]$ and $\mu \in \cP(\R_\plus)$.
Then $[\mu]_c < \infty$ if and only if $\mu \in \PHS$.
\end{proposition}

\begin{proof}
The case $c=1$ being clear, assume that $c \in (0,1)$.
Note that for all $y>0$, the expression
\begin{equation}
\left | K(x,y,c) -  c^{-1} \log(x+1) \right| \,
\end{equation}
is a bounded function of $x$,
so the integrability of $K(x,y,c)$ and $\log(x+1)$ are equivalent. 
\end{proof}

Next, let us see that the standard properties one might expect for something called a ``barycenter'' are satisfied.
For any $x \ge 0$, we denote by $\delta_x$ the probability measure such that $\delta_x(\{x\})=1$. %{\color{magenta} Already defined Dirac delta at the beginning of the section}
%\margin{A previous defintion of delta was removed...}

\begin{proposition} %[Basic properties of the HS barycenter] 
\label{p.basic}
For all $c \in [0,1]$ and $\mu \in \PHS$, the following properties hold:
\begin{enumerate}

\item\label{i.basic_reflex} 
\emph{Reflexivity:} $[\delta_x]_c = x$, for every $x \ge 0$.

\item\label{i.basic_mono_mu} 
\emph{Monotonicity with respect to the measure:} If $\mu_1$, $\mu_2 \in \PHS$ have distribution functions $F_1$, $F_2$ such that $F_1 \ge F_2$\footnote{I.e., $\mu_2$ is ``more to the right'' than $\mu_1$. This defines a partial order, called \emph{usual stochastic ordering} or \emph{first order stochastic dominance}.}, then $[\mu_1]_c \le [\mu_2]_c$.

\item\label{i.basic_intern} 
\emph{Internality:} If $\mu(I)=1$ for an interval $I \subseteq \R_\plus$, then $[\mu]_c \in I$.
%\margin{Minor change (internality).}

\item\label{i.basic_homog} 
\emph{Homogeneity:} If  $\lambda \ge 0$, and $\lambda_* \mu$ denotes the pushforward of $\mu$ under the map $x \mapsto \lambda x$, then $[\lambda_* \mu]_c = \lambda [\mu]_c$.

\item
\label{i.basic_mono_c} \emph{Monotonicity with respect to the parameter:} 
If $0 \le c' \le c \le 1$, then $[\mu]_{c'} \ge [\mu]_{c}$. 

\end{enumerate}
\end{proposition}

\begin{proof}
% It suffices to consider $0<c<1$.
The proofs use the properties of the HS kernel listed in \cref{p.kernel}.
Reflexivity is obvious when when $c=1$ or $x=0$, and in all other cases follows from property \eqref{i.kernel_reflex}. 
Monotonicity with respect to the measure %(resp.\ to $c$) 
is a consequence of the fact that the HS kernel is increasing in $x$. 
% (resp.\ decreasing in $c$). 
The internality property of the HS barycenter follows from reflexivity and monotonicity.
Homogeneity follows from property \eqref{i.kernel_homog} of the HS kernel and the change of variables formula.
Finally, monotonicity with respect to the parameter $c$ is a consequence of the corresponding property of the HS kernel.
\end{proof}

As it will be clear later (see \cref{ex.Bernoulli}), the internality and the monotonicity properties (w.r.t.\ $\mu$ and w.r.t.\ $c$) are not strict.

\subsection{Computation and critical phenomenon}
%\margin{New subsection}

In the remaining of this section, we discuss how to actually compute HS barycenters. 
In view of \cref{p.HS_infinity}, we may focus on measures in $\PHS$.
The mass of zero plays an important role. % in the behavior of the HS barycenter.
Given $c \in (0,1)$ and $\mu \in \PHS$,  we use the following terminology, where $\mu(0) = \mu(\{0\})$:
\begin{equation}\label{e.def_cases}
\left.
\begin{tabular}{l r}
subcritical case:        & $\mu(0) < 1-c$ \\ 
critical case:           & $\mu(0) = 1-c$ \\ 
supercritical case:      & $\mu(0) > 1-c$ 
\end{tabular}
\qquad
\right\}
\end{equation}
The next result establishes a way to compute $[\mu]_c$ in the subcritical case; the remaining cases will be dealt with later in \cref{p.critical_supercritical}.
%{\color{magenta} Do we need to say that $K_y$ is the derivative?}

\begin{proposition}\label{p.subcritical} 
If $\mu \in \PHS$, $c \in (0, 1)$
and $\mu(0) < 1-c$ (subcritical case), then the equation
\begin{equation}\label{e.eta_abstract}
\int K_y (x,\eta,c) \, d \mu(x) = 0
\end{equation}
has a unique positive and finite solution $\eta = \eta(\mu,c)$, 
and the $\inf$ in formula \eqref{e.def_HS} is attained uniquely at $y=\eta$; in particular,
\begin{equation}\label{e.sub_conclusion}
[\mu]_c = \exp \int K(x, \eta, c) \, d \mu(x).
\end{equation}
\end{proposition}

\begin{proof}
Fix $\mu$ and $c$ as in the statement.
We compute the partial derivative:
\begin{equation}\label{e.Delta}
K_y (x,y,c) = \frac{\Delta(x,y)}{y} \, , \quad \text{where} \quad
\Delta(x,y) \coloneqq  1 - \frac{x}{cx+(1-c)y} \, .
\end{equation}
Since $\Delta$ is bounded, we are allowed to differentiate under the integral sign:
\begin{equation}\label{e.silly}
\frac{d}{dy} \int K(x,y,c) \, d\mu 
= \int K_y (x,\eta,c) \, d \mu \\
= \frac{\psi(y)}{y} \, , 
\end{equation}
where $\psi(y) \coloneqq \int \Delta(x,y) \, d\mu$.
The partial derivative
\begin{equation}\label{e.Delta_y}
\Delta_y(x,y) =  \frac{(1-c)x}{(cx+(1-c)y)^2} 
\end{equation}
is positive, except at $x=0$.
Since $\mu \neq \delta_0$, the function $\psi$ is strictly increasing.
Furthermore,
\begin{equation}
\lim_{y \to +\infty} \Delta(x,y) = 1  \quad \text{and} \quad
\lim_{y \to 0^+} \Delta(x,y) = 
\begin{cases}
1-1/c &\text{if $x>0$},\\
1 &\text{if $x=0$},
\end{cases}
\end{equation}
and so
\begin{equation}
\lim_{y \to +\infty} \psi(y) = 1  \quad \text{and} \quad
\lim_{y \to 0^+} \psi(y) =  1 - \frac{1-\mu(0)}{c} < 0 \, ,
\end{equation}
using the assumption $\mu(0)<1-c$.
Therefore there exists a unique $\eta>0$ that solves the equation $\psi(\eta)=0$, or equivalenlty equation \eqref{e.eta_abstract}. 
By \eqref{e.silly}, the function $y \mapsto \int K(x,y,c) \, d\mu$ decreases on $(0,\eta]$ and increases on $[\eta,+\infty)$, and so attains its infimum at $\eta$. Formula \eqref{e.sub_conclusion} follows from the definition of $[\mu]_c$.
\end{proof}

Let us note that, as a consequence of \eqref{e.Delta}, equation~\eqref{e.eta_abstract} is equivalent to:
\begin{equation}\label{e.eta}
\int \frac{x}{cx+(1-c)\eta} \, d \mu(x) = 1 \, .
\end{equation}

\begin{remark}\label{r.eta_other_cases}
If $\mu$ belongs to $\cP(\R_\plus)$ but not to $\PHS$, and still $c \in (0,1)$, then equation~\eqref{e.eta_abstract} (or its equivalent version \eqref{e.eta}) still has a unique positive and finite solution $\eta = \eta(\mu,c)$, and formula \eqref{e.sub_conclusion} still holds.
On the other hand, if $c=0$ and $\int x \, d\mu(x) < \infty$, then all conclusions of  \cref{p.subcritical} still hold, with a similar proof.
\end{remark}

We introduce the following auxiliary function, plotted in \cref{fig.B}:
\begin{equation} \label{e.def_B}
B(c) \coloneqq 
\begin{cases}
0 &\text{if } c=0, \\ 
c(1-c)^\frac{1-c}{c} &\text{if } 0<c<1, \\ 
1 &\text{if } c=1.
\end{cases}
\end{equation}
%Note that $B$ is a continuous strictly increasing function on the interval $[0,1]$.

% Maple code:
% plot(c*(1-c)^((1-c)/c),c=0..1,labels=[c,B(c)],labelfont=[ROMAN,italic,10]);
\begin{figure}[htb]
\begin{center}
\includegraphics[width=.5\textwidth]{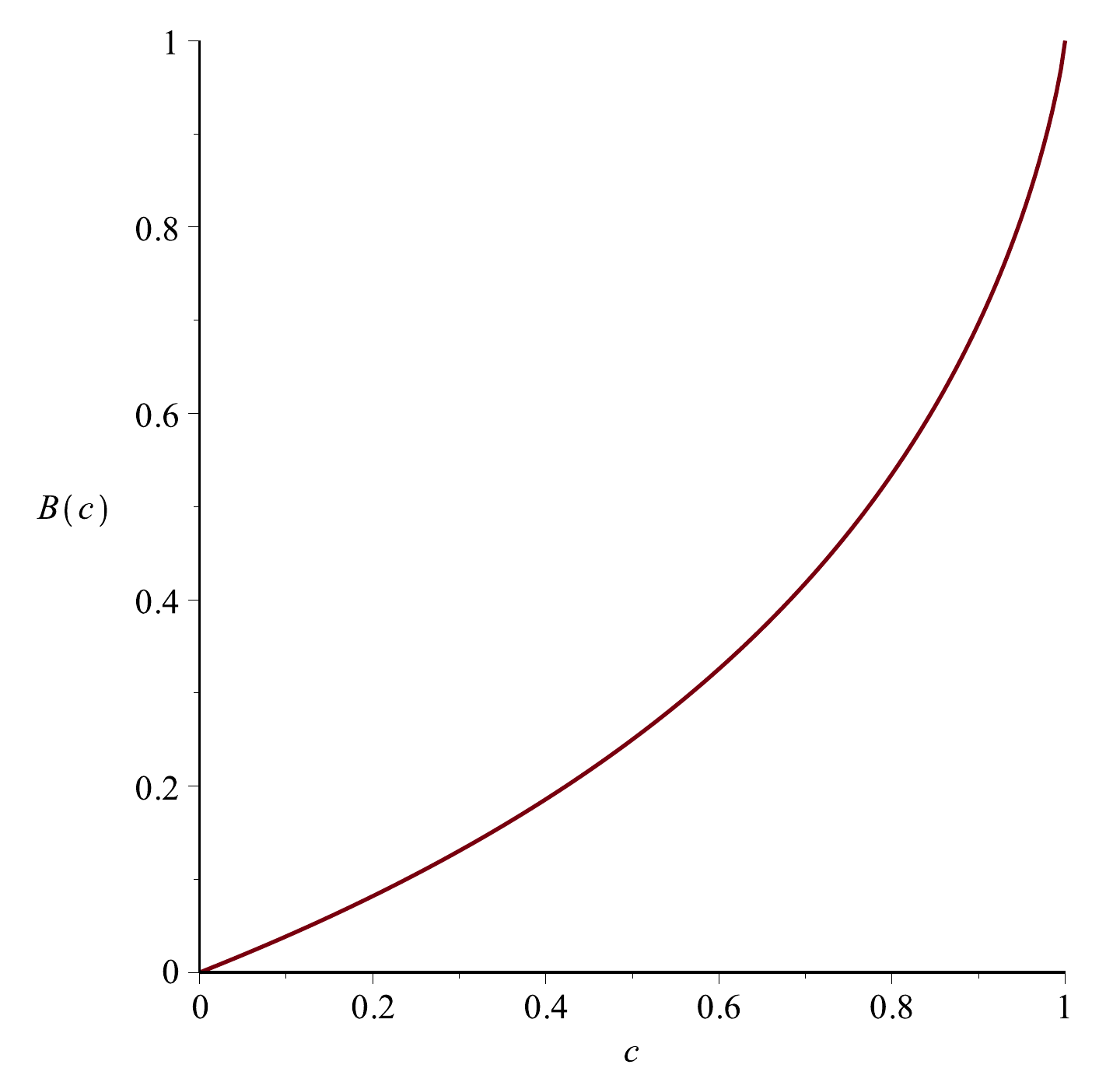}
\caption{Graph of the function $B$ defined by \eqref{e.def_B}.}\label{fig.B}
\end{center}
\end{figure}

The following alternative formula for the HS barycenter matches the original one  from \cite{HS76}, and in some situations is more convenient: 
\begin{proposition}\label{p.oldstyle}
If $0 < c\le 1$ and $\mu\in \PHS$, then:
\begin{equation}\label{e.oldstyle}
[\mu]_c = B(c) \, \inf_{r>0} \left\{ r^{-\frac{1-c}{c}} \exp \left[ c^{-1} \int \log (x+r) \, d\mu(x) \right] \right\} \, .
\end{equation}
Furthermore, if $\mu(0)<1-c$, then the $\inf$ is attained at the unique positive finite solution $\rho = \rho(\mu,c)$ of the equation 
\begin{equation}\label{e.rho}
\int \frac{x}{x+\rho} \, d\mu(x) = c \, .
\end{equation}
\end{proposition}

\begin{proof}
The formula is obviously correct if $c=1$.
If $0<c<1$, we introduce the variable $r \coloneqq \frac{1-c}{c} \, y$ in formula \eqref{e.def_HS} and manipulate. 
Similarly, \eqref{e.eta} becomes \eqref{e.rho}.
\end{proof}

If $\mu$ is a Borel probability measure on $\R_{\plus}$ not entirely concentrated at zero (i.e., $\mu \neq \delta_0$) then we denote by  $\mu^\plus$  the probability measure obtained by conditioning on the event $\R_{\plus\plus} =(0,\infty)$, that is, 
%\margin{I repeated the definition of $\R_{\plus\plus}$ as a reminder, but if you prefer to delete it, go ahead.}
\begin{equation}\label{e.conditional}
\mu^\plus(U) \coloneqq \frac{\mu(U \cap \R_{\plus\plus})}{\mu(\R_{\plus\plus})}, \quad \text{for every Borel set } U \subseteq \R \, .
\end{equation}
Obviously, if $\mu \in \PHS$, then $\mu^\plus \in \PHS$ as well.

\begin{proposition} \label{p.induce}
Let $\mu \in \PHS$     and  $p \coloneqq 1-\mu(0)$.
If $c \in (0,1]$ and $c \le p$ (critical or subcritical cases), then
\begin{equation}\label{e.induce}
[\mu]_c= \frac{B(c)}{B(c/p)} [\mu^\plus]_{c/p} \, .
\end{equation}
\end{proposition}

\begin{proof}
Note that $\mu \neq \delta_0$, so the positive part $\mu^\plus$ is defined. 
Using formula \eqref{e.oldstyle}, we have:
\begin{align}
[\mu]_c
&= B(c) \inf_{r >0} r^{1-\frac{1}{c}} \exp \left(\frac{1}{c} \int \log (x+r) d \mu		\right)  \\
\displaybreak[0]
&= B(c) \inf_{r >0} r^{1-\frac{1}{c}} \exp \left(\frac{1}{c} \left( \mu(0) \log r +  \int_{\{x>0\} } \log (x+r) d \mu \right) \right) \\
\displaybreak[0]
&= B(c) \inf_{r >0} r^{1-\frac{1}{c}+ \frac{1-p}{c}} \exp \left(\frac{1}{c} \int_{\{x>0\} } \log (x+r) d \mu		\right) \\
\displaybreak[0]
&= B(c) \inf_{r >0} r^{1- \frac{p}{c}} \exp \left(\frac{1}{c/p}  \int \log(x+r) d \mu^\plus		\right) \, . \label{e.almost_induce}
% \\
% & = \frac{B(c)}{B(c/p)} [\mu^\plus]_{c/p} \, ,
\end{align}
At this point, the assumption $c \le p$ guarantees that the barycenter $[\mu^\plus]_{c/p}$ is well-defined, and using \eqref{e.oldstyle} again we obtain \eqref{e.induce}.
\end{proof}

Finally, we compute the HS barycenter in the critical and supercritical cases:
\begin{proposition}\label{p.critical_supercritical}
Let $\mu \in \PHS$ and $c \in (0,1]$.
\begin{enumerate}
\item\label{i.critical} 
\emph{Critical case:} If $\mu(0)=1-c$, then $[\mu]_c = B(c) [\mu^\plus]_1$. %, where $\mu^\plus$ is the conditional probability \eqref{e.conditional}.
\item\label{i.supercritical}  
\emph{Supercritical case:} If $\mu(0)>1-c$, then $[\mu]_c = 0$.
\end{enumerate}
In both cases above, the infimum in formula \eqref{e.def_HS} is \emph{not} attained.
\end{proposition}

\begin{proof}
In the critical case, we use \eqref{e.induce} with $p=c$ and conclude.

In the supercritical case, we can assume that $\mu \neq \delta_0$.
Note that $p<c$, thus $\lim_{r \to 0^+} r^{1-\frac{p}{c}}=0$. 
Moreover, since $\log^+ x \in L^1(\mu)$, we have
\begin{equation}
\lim_{r \to 0^+} \int_{\{x>0\}} \log(x+r) \, d\mu 
= \int_{\{x>0\}} \log x \, d\mu 
< +\infty \, .
\end{equation}
Therefore, using \eqref{e.almost_induce}, we obtain $[\mu]_c = 0$. 
\end{proof}

\cref{p.subcritical,p.critical_supercritical}  allow us to compute HS barycenters in all cases. For emphasis, let us list explicitly the situations where the barycenter vanishes: 

%\margin{Kept the proposition...}

\begin{proposition}\label{p.zeros}
Let $c \in [0,1]$ and $\mu \in \PHS$. 
Then $[\mu]_c = 0$ if and only if one of the following mutually exclusive situations occur:
\begin{enumerate}
\item $c = 0$ and $\mu = \delta_0$.
\item $c>0$, $\mu(0) = 1-c$, and %the conditional probability $\mu^\plus$ defined by \eqref{e.conditional} satisfies 
$\int \log x \, d\mu^\plus(x) = -\infty$.
\item $c>0$ and $\mu(0) > 1-c$.
\end{enumerate}
\end{proposition}

\begin{proof}
The case $c=0$ being obvious, assume that $c>0$, and so $B(c)>0$.
In the critical case, part \eqref{i.critical} of \cref{p.critical_supercritical} tells us that $[\mu]_c = 0$ if and only if $[\mu^\plus]_1 = 0$, which by \eqref{e.HS1} is equivalent to $\int \log x \, d\mu^\plus = -\infty$. 
In the supercritical case, part \eqref{i.supercritical} of the \lcnamecref{p.critical_supercritical}  ensures that $[\mu]_c = 0$.
\end{proof}

\begin{example}\label{ex.Bernoulli}
Consider the family of probability measures:
\begin{equation}
\mu_p \coloneqq (1-p) \delta_0 + p \delta_1 \, ,\quad 0\le p \le 1.
\end{equation}
If $0<c<1$, then  %, using \cref{p.induce} and 
\begin{equation}\label{e.one_param}
[\mu_p]_c = 
\begin{cases} 
	0 &\quad \text{if } p<c \\ 
	B(c) = c (1-c)^{\frac{1-c}{c}} &\quad \text{if } p=c \\ 
	\frac{B(c)}{B(c/p)} = (1-c)^{\frac{1-c}{c}} p^{\frac{p}{c}} (p-c)^{\frac{c-p}{c}}  &\quad \text{if } p>c \, .  
\end{cases}
\end{equation}
These formulas were first obtained by Sz\'ekely \cite{Szekely_first}.
% See \cref{fig.2graphs} below for a plot with $c=1/2$.
It follows that the function 
\begin{equation}
(p,c) \in [0,1] \times [0,1] \mapsto [\mu_p]_c
\end{equation}
(whose graph is shown on \cite[p.~680]{vanEs})
is discontinuous at the points with $p=c>0$, and only at those points.
We will return to the issue of continuity in \cref{s.continuity}.
\end{example}

% {\color{blue} **Write down the (ugly) formula for $[p \delta_x + q \delta_y]_c$? The particular case $p=q=c=1/2$ is mentioned in the introduction and in \cref{ex.Cellarosi}.**} 

%\margin{Summarized this part.}
The practical computation of HS barycenters usually requires numerical methods.
In any case, it is useful to notice that the function $\eta$ from \cref{p.subcritical} satisfies the internality property:

\begin{lemma}\label{l.internality_eta}
Let $\mu \in \PHS$, $c \in (0, 1)$, and suppose that $\mu(0) < 1-c$.
If $\mu(I)=1$ for an interval $I \subseteq \R_\plus$, then $\eta(\mu,c) \in I$.
\end{lemma}

The proof is left to the reader.

% Let us discuss the practical computation of HS barycenters.
% The cases $c=0$ and $c=1$ deserve no comment, so suppose that $0<c<1$.
% As we have seen in \cref{p.critical_supercritical}, the critical and supercritical cases are essentially trivial.
% In the subcritical case, one can compute the HS barycenter using \cref{p.subcritical} (or, equivalently, \cref{p.oldstyle}). This computation usually requires numerical methods.
% the following remark may be useful: If one can find an interval $I = [a,b]$ such that $\mu(I) = 1$ (i.e.,  the probability measure $\mu$ has compact support), then the solution $\eta$ of equation \eqref{e.eta} also belongs to $I$; this is a consequence of the following fact:

% \begin{lemma}\label{l.internality_eta}
% Suppose that $\mu \in \PHS$ is not a delta measure, $c>0$, and $\mu(0) < 1-c$, and let $\eta$ be the positive solution of equation~\eqref{e.eta}. Then:
% \begin{equation}
% \mu \big( [0,\eta] \big) < 1 \quad \text{and} \quad
% \mu \big( [\eta,+\infty) \big) < 1 \, .
% \end{equation}
% \end{lemma}

% \begin{proof}
% If either $[0,\eta]$ or $[\eta,+\infty)$ has full measure, then the integrand in \eqref{e.eta} is everywhere constant equal to $1$, which implies that $\mu = \delta_{\eta}$.
% \end{proof}

%%%%%%%%%%%%%%%%%%%%%%%%%%%%%%%%%%%%%%%%%%%%%%%%%%%%%%%%%%%%%%%%%%%%%%%%%%%%%%%%%%%
\section{Comparison with the symmetric means}\label{s.inequalities}
%%%%%%%%%%%%%%%%%%%%%%%%%%%%%%%%%%%%%%%%%%%%%%%%%%%%%%%%%%%%%%%%%%%%%%%%%%%%%%%%%%%

\subsection{HS means as repetitive symmetric means}\label{ss.repetitive}

The HS barycenter of a probability measure, introduced in the previous section, may now be specialized to the case of discrete equidistributed probabilities.
So the \emph{HS mean} of a tuple $\underline{x} = (x_1,\dots,x_n)$ of nonnegative numbers with parameter $c \in [0,1]$ is defined as:
\begin{equation}\label{e.def_HS_mean}
\hsm_c(\underline{x}) \coloneqq 
\left[ \frac{\delta_{x_1} + \dots + \delta_{x_n}}{n}\right]_c \, .
\end{equation}
% where $[ \mathord{\cdot} ]_c$ denotes the HS barycenter of a probability measure, introduced in the previous section.
Using \eqref{e.oldstyle}, we have more explicitly:
\begin{equation}\label{e.hsm_formula}
\hsm_c(\underline{x}) =
B(c)
\inf_{r>0} r^{-\frac{1-c}{c}} \prod_{j=1}^n (x_j + r)^\frac{1}{cn} 
\quad \text{(if $c>0$)},
\end{equation}
where $B$ is the function \eqref{e.def_B}.
On the other hand, recall %from the introduction 
that for $k \in \{1,\dots,n\}$, the \emph{$k$-th symmetric mean} of the $n$-tuple $\underline{x}$ is:
\begin{equation}\label{e.def_sym_mean_again}
\sym_k(\underline{x}) \coloneqq \left(\frac{E^{(n)}_k(\underline{x})}{\binom{n}{k}}\right)^{\frac{1}{k}} \, ,
\end{equation}
where $E^{(n)}_k$ denotes the elementary symmetric polynomial of degree $k$ in $n$ variables.

Since they originate from a barycenter, the HS means are \emph{repetition invariant}\footnote{or \emph{intrinsic}, in the terminology of \cite[Def.~3.3]{KLL}} in the sense that, for any $m>0$, 
\begin{equation}\label{e.duh}
\hsm_c \left( \underline{x}^{(m)} \right)  = \hsm_c ( \underline{x} ) \, , 
\end{equation}
where $\underline{x}^{(m)}$ denotes the $nm$-tuple obtained by concatenation of $m$ copies of the $n$-tuple $\underline{x}$.
No such property holds for the symmetric means, even allowing for adjustment of the parameter $k$.
Nevertheless, if the number of repetitions tends to infinity, then the symmetric means tend to stabilize, and the limit is a HS mean; more precisely:

\begin{theorem}\label{t.repeat}
If $\underline{x} = (x_1,\dots,x_n)$, 
$x_i \ge 0$, and $1 \le k \le n$, then:
\begin{equation}\label{e.repeat}
\lim_{m \to \infty} 
\sym_{km} \left( \underline{x}^{(m)} \right) 
%\sym_{km} \big( \underbrace{x_1, \dots, x_1}_{\text{$m$ times}}, \,  \dots, \, \underbrace{x_n, \dots, x_n}_{\text{$m$ times}} \big) 
= \hsm_{k/n} (\underline{x}) \, .
\end{equation}
Furthermore, the relative error goes to zero uniformly with respect to the $x_i$'s.
\end{theorem}

This \lcnamecref{t.repeat} will be proved in the next subsection. 
% {\color{blue} And an even more general result will be obtained later: \cref{t.ergodic} (but without the uniformity...)}

It is worthwhile to note that the Navas barycenter \cite{Navas} %(on a metric space of nonpositive curvature) 
is obtained as a ``repetition limit'' similar to \eqref{e.repeat}.

\begin{example}\label{ex.Cellarosi}
Using \cref{p.subcritical,p.critical_supercritical}, one computes:
\begin{equation}
\hsm_{1/2} (x,y) = \left(\frac{\sqrt{x}+\sqrt{y}}{2}\right)^2 \, ;
\end{equation}
Therefore:
\begin{equation}
\lim_{m \to \infty} \sym_{m} \big( \underbrace{x, \dots, x}_{\text{$m$ times}}, \,  \underbrace{y, \dots, y}_{\text{$m$ times}} \big) =  \left(\frac{\sqrt{x}+\sqrt{y}}{2}\right)^2 \, .
\end{equation}
The last equality was deduced in \cite[p.~31]{Cellarosi} from the asymptotics of Legendre polynomials. 
\end{example}

Let us pose a problem:

\begin{question}
Is the sequence $m \mapsto \sym_{km} \left( \underline{x}^{(m)} \right)$ always monotone decreasing?
\end{question}

There exists a partial result: when $k=1$ and $n=2$, \cite[Lemma~4.1]{Cellarosi} establishes \emph{eventual} monotonicity.

% {\color{red} 
% {\bf Question: Is the sequence $m \mapsto \sym_{km} \left( \underline{x}^{(m)} \right)$ monotone?} \cite[Lemma~4.1]{Cellarosi} proves \textbf{eventual} monotonicity when $k=1$ and $n=2$.\footnote{ BTW, they mention a lack of monotonicity when $c\neq 1/2$, but I think that this only an unnatural side effect of a rounding: see the term $\binom{2k}{\lceil 2ck \rceil}$ in tehir formulas.} I couldn't find anything in the book \cite{Bullen}. BUT MAYBE THE PAPER \cite{Bullen_65} IS RELEVANT; CHECK. (The paper is in Dropbox.)  GI: We tried and failed.
% \medskip}

\subsection{Inequalities between symmetric means and HS means}\label{ss.comparison}

The following is the first main result of this paper. 
%\margin{Added this sentence.}

\begin{theorem}\label{t.main}
If $\underline{x} = (x_1,\dots,x_n)$, 
%is a tuple of nonnegative numbers,
$x_i \ge 0$, and $1 \le k \le n$, then
\begin{equation}\label{e.main}
\hsm_{k/n}(\underline{x}) \le 
\sym_{k}(\underline{x}) \le 
\frac{\binom{n}{k}^{-1/k}}{B \left( \frac{k}{n} \right)}  \, \hsm_{k/n}(\underline{x}) \, .
\end{equation}
\end{theorem}

Let us postpone the proof to the next subsection. % and extract the first consequences immediately.
The factor at the RHS of \eqref{e.main} is asymptotically $1$ with respect to $k$; indeed: 

%\margin{Brought the lemma (and a remark) to here, since they were kind of floating in the end of the section.}

\begin{lemma}\label{l.asymp}
For all integers $n \ge k \ge 1$, we have 
\begin{equation}\label{e.asymptotics}
1 \le \frac{\binom{n}{k}^{-1/k}}{B \left( \frac{k}{n} \right)} < (9k)^\frac{1}{2k}  \, ,
\end{equation}
with equality if and only if $k=n$.
\end{lemma}

\begin{proof} %[Proof of bounds \eqref{e.asymptotics}]
Let $c \coloneqq k/n$ and 
\begin{equation}
b \coloneqq \binom{n}{k} \left[B(c)\right]^k = \binom{n}{k} c^k \, (1-c)^{n-k} \, .
\end{equation}
By the Binomial Theorem, $b \le 1$, which yields the first part of \eqref{e.asymptotics}.
For the lower estimate, we use the following Stirling bounds (see \cite[p.~54, (9.15)]{Feller}), valid for all $n\ge 1$,
\begin{equation}
\sqrt{2\pi} \, n^{n+\frac{1}{2}} e^{-n} < n! < \sqrt{2\pi} \, n^{n+\frac{1}{2}} e^{-n+\frac{1}{12}} \, .
\end{equation}
Then, a calculation (cf. \cite[p.~184, (3.11)]{Feller}) gives:
\begin{equation}
b > \frac{e^{-\frac{1}{6}}}{\sqrt{2 \pi c(1-c) n}}  > \frac{1}{\sqrt{9k}} \, ,
\end{equation}
from which the second part of \eqref{e.asymptotics} follows.
\end{proof}

\cref{t.main,l.asymp} imply that HS means (with rational values of the parameter) can be obtained as repetition limits of symmetric means:

%\margin{Summarized the proof}

\begin{proof}[Proof of \cref{t.repeat}]
Applying \cref{t.main} to the tuple $\underline{x}^{(m)}$, using observation \eqref{e.duh} and \cref{l.asymp}, we have:
\begin{equation}
   \hsm_{k/n}(\underline{x}) \le 
\lim_{m \to \infty}\sym_{km}(\underline{x}^{(m)}) \le 
% \frac{\binom{nm}{km}^{-1/km}}{B \left( \frac{km}{nm} \right)}
(9km)^\frac{1}{2km} \, 
\hsm_{k/n}(\underline{x})  \,  .   \qedhere
\end{equation}
%By \cref{l.asymp}, the fraction in the RHS converges to $1$ as $m \to \infty$, so the result follows.
% 
% From \cref{t.main} we have  that for every $m \in \N$ the following inequalities hold:
% \begin{equation}
%   \hsm_{km/nm}(\underline{x}^{(m)}) \le 
% \sym_{km}(\underline{x}^{(m)}) \le 
% \frac{\binom{nm}{km}^{-1/km}}{B \left( \frac{km}{nm} \right)}  \, \hsm_{km/nm}(\underline{x}^{(m)}) \, .  
% \end{equation}
% As observed in \eqref{e.duh}, $\hsm_{km/nm}(\underline{x}^{(m)})=\hsm_{k/n}(\underline{x})$. Therefore,
% \begin{equation}
%   \hsm_{k/n}(\underline{x}) \le 
% \lim_{m \to \infty}\sym_{km}(\underline{x}^{(m)}) \le 
% \hsm_{k/n}(\underline{x}) \,  \lim_{m \to \infty} \frac{\binom{nm}{km}^{-1/km}}{B \left( \frac{km}{nm} \right)}  \,  .  
% \end{equation}
% Since the limit in the LHS converges to $1$ (see \cref{l.asymp}), the result follows.
\end{proof}

\begin{remark}\label{r.Maclaurin}
If $k$ is fixed, then 
\begin{equation}
\lim_{n \to \infty} \frac{\binom{n}{k}^{-1/k}}{B \left( \frac{k}{n} \right)} = \frac{e (k!)^{1/k}}{k} > 1 \, ,
\end{equation}
and therefore the bound from \cref{t.main_upper} may be less satisfactory.
But in this case we may use the alternative bound coming from Maclaurin inequality:
\begin{equation}\label{e.Maclaurin_particular}
\sym_k(\underline{x}) \le \sym_1(\underline{x}) = \hsm_0(\underline{x}) \, . 
\end{equation}
\end{remark}

%\margin{New subsection}
\subsection{Proof of \cref{t.main}}

The two inequalities in \eqref{e.main} will be proved independently of each other. They are essentially contained in the papers \cite{BIP} and \cite{HS76}, respectively, though neither was stated explicitly.
In the following \cref{t.main_lower,t.main_upper}, we also characterize the cases of equality, and in particular show that each inequality is sharp in the sense that the corresponding factors cannot be improved.

\medskip

Let us begin with the second inequality, which is more elementary.
By symmetry, there is no loss of generality in assuming that the numbers $x_i$ are ordered. % in non-increasing fashion.

\begin{theorem}\label{t.main_upper}
If $\underline{x} = (x_1,\dots,x_n)$ with $x_1 \ge \cdots \ge x_n \ge 0$, and $1 \le k \le n$, then:
\begin{equation}\label{e.main_upper}
\sym_k(\underline{x})
\le 
\frac{\binom{n}{k}^{-1/k}}{B \left( \frac{k}{n} \right)}
\hsm_{k/n}(\underline{x}) \, .
%\left[ \frac{\delta_{x_1} + \dots + \delta_{x_n}}{n}\right]_{k/n} \, .
\end{equation}
Furthermore, equality holds if and only if $x_{k+1} = 0$ or $k=n$.
\end{theorem}

\begin{proof}
Our starting point is Vieta's formula: 
\begin{equation}
z^n + \sum_{\ell = 1}^n E^{(n)}_\ell(x_1, \dots, x_n) z^{n-\ell} = \prod_{j=1}^n (x_j+z) \, .
\end{equation}
Therefore, by Cauchy's formula, for any $r>0$:
\begin{equation} \label{e.Cauchy_complex_var}
E^{(n)}_k(x_1, \dots, x_n) = \frac{1}{2\pi \mathbf{i}} \ointctrclockwise_{|z|=r} \frac{1}{z^{n-k+1}} \prod_{j=1}^n (x_j+z) \, dz \, .
\end{equation}
That is,
\begin{equation}\label{e.Cauchy_real_var}
E^{(n)}_k(x_1, \dots, x_n) = \frac{1}{2\pi} \int_{-\pi}^{\pi}(r e^{\mathbf{i}\theta})^{-n+k} \prod_{j=1}^n (x_j + r e^{\mathbf{i}\theta})  \, d\theta
\, ,
\end{equation}
Taking absolute values,
\begin{equation}\label{e.abs_value}
E^{(n)}_k(x_1, \dots, x_n) 
\le 
\frac{1}{2\pi} \int_{-\pi}^{\pi} r^{-n+k} \prod_{j=1}^n |x_j + r e^{\mathbf{i}\theta}|  \, d\theta 
\le r^{-n+k} \prod_{j=1}^n (x_j + r) \, .
\end{equation}
% Since $|x_j + r e^{\mathbf{i}\theta}| =  \big( x_j^2 + 2 r x_j \cos \theta + r^2 \big)^{1/2} \le x_j+r$, the absolute value of the integrand attains its maximum at $\theta = 0$.
But these inequalities are valid for all $r>0$, and therefore:
\begin{equation}\label{e.nice_ineq}
E^{(n)}_k(x_1, \dots, x_n) \le \inf_{r>0} r^{-n+k} \prod_{j=1}^n (x_j + r) \, .
\end{equation}
% That is, 
% \begin{equation}
% [E^{(n)}_k(x_1, \dots, x_n)]^{\frac{1}{k}} \le \inf_{r>0} r^{-\frac{n-k}{k}} \prod_{j=1}^n (x_j + r)^{\frac{1}{k}} \, ,
% \end{equation}
So formulas \eqref{e.hsm_formula} and \eqref{e.def_sym_mean_again} imply inequality \eqref{e.main_upper}.

Now let us investigate the possibility of equality. 
We consider three mutually exclusive cases, which correspond to the classification \eqref{e.def_cases}:
\begin{equation}
\left.
\begin{tabular}{l l}
subcritical case:        & $x_{k+1}>0$ \\ 
critical case:           & $x_k > 0 = x_{k+1}$ \\ 
supercritical case:      & $x_k = 0$ 
\end{tabular}
\qquad
\right\}
\end{equation}
Using \cref{p.critical_supercritical},
in the critical case we have:
\begin{equation}
\sym_k (\underline{x}) = \left(\frac{x_1\cdots x_k}{\binom{n}{k}}\right)^{\frac{1}{k}} 
\quad \text{and} \quad
\hsm_{k/n}(\underline{x}) = \frac{(x_1\cdots x_k)^{\frac{1}{k}}}{B(\frac{k}{n})}  \, ,
\end{equation}
while in the supercritical case the two means vanish together. 
So, in both cases, %the supercritical and the critical cases, 
inequality \eqref{e.main_upper} becomes an equality. 
Now suppose we are in the subcritical case; then the $\inf$ at the RHS of \eqref{e.nice_ineq} is attained at some $r > 0$: see \cref{p.oldstyle}.
On the other hand, for this (and actually any) value of $r$, the second inequality in \eqref{e.abs_value} must be strict, because the integrand is non-constant. 
We conclude that, in the subcritical case, inequality \eqref{e.nice_ineq} is strict, and therefore \eqref{e.main_upper} is strict.
% So we have proved that $x_{k+1} = 0$ if and only if \eqref{e.main_upper} becomes an equality.
\end{proof}

\medskip

The first inequality in \eqref{e.main} is a particular case of an inequality between two types of \emph{matrix means} introduced in \cite{BIP}, which we now explain. Let $A = (a_{i,j})_{i,j\in\{1,\dots,n\}}$ be a $n \times n$ matrix with nonnegative entries. Recall that the \emph{permanent} of $A$ is the ``signless determinant''
\begin{equation}\label{e.def_per}
\per(A) \coloneqq \sum_{\sigma} \prod_{i=1}^n a_{i,\sigma(i)} \, ,
\end{equation}
where $\sigma$ runs on the permutations of $\{1, \dots, n\}$.
Then the \emph{permanental mean} of $A$ is defined as:
\begin{equation}\label{e.def_p,}
\pmean(A) \coloneqq \left( \frac{\per(A)}{n!} \right)^{\frac{1}{n}} \, .
\end{equation}
On the other hand, the \emph{scaling mean} of the matrix $A$ is defined as:
\begin{equation}\label{e.def_sm}
\smean(A) \coloneqq \frac{1}{n^2} \inf_{u,v} \frac{u^\top A v}{\gmean(u)\gmean(v)} \, ,
\end{equation}
where $u$ and $v$ run on the set of strictly positive column vectors,
and $\gmean(\mathord{\cdot})$ denotes the geometric mean of the entries of the vector.
Equivalently,
\begin{equation}\label{e.sm_inf_formula}
\smean(A) = \frac{1}{n} \inf_{v} \frac{\gmean(Av)}{\gmean(v)} \, ;
\end{equation}
see \cite[Rem.~2.6]{BIP}.\footnote{Incidentally, formula \eqref{e.sm_inf_formula} shows that, up to the factor $1/n$, the scaling mean is a \emph{matrix antinorm} in the sense defined by \cite{GZ}.}
By \cite[Thrm.~2.17]{BIP},
\begin{equation}\label{e.our_vdW}
\smean(A) \le \pmean(A) \, ,
\end{equation}
with equality if and only if $A$ has permanent $0$ or rank $1$.
This inequality is far from trivial.
Indeed, if the matrix $A$ is doubly stochastic (i.e.\ row and column sums are all $1$), then an easy calculation (see \cite[Prop.~2.4]{BIP}) shows that $\smean(A) = \frac{1}{n}$, so \eqref{e.our_vdW} becomes $\pmean(A) \ge \frac{1}{n}$, or equivalently, 
\begin{equation}\label{e.vdW}
\per(A) \ge \frac{n!}{n^n} \quad \text{(if $A$ is doubly stochastic).}
\end{equation}
This lower bound on the permanent of doubly stochastic matrices was conjectured in 1926 by van~der~Waerden and, after a protracted series of partial results, proved around 1980 independently by Egorichev and Falikman: see \cite[Chapter~5]{Zhan} for the exact references and a self-contained proof, and \cite{Gurvits} for more recent developments. 
Our inequality \eqref{e.our_vdW}, despite being a generalization of Egorichev--Falikman's \eqref{e.vdW}, is actually a relatively simple corollary of it: we refer the reader to \cite[\S~2]{BIP} for more information.\footnote{The proof of \cite[Thrm.~2.17]{BIP} uses a theorem on the existence of particular type of matrix factorization called Sinkhorn decomposition. The present article only needs the inequality \eqref{e.our_vdW} for matrices of a specific form \eqref{e.special}. So the use of the existence theorem could be avoided, since it is possible to explicitly compute the corresponding Sinkhorn decomposition.}
%\margin{Minor change in the footnote.}

\medskip

We are now in position to complete the proof of \cref{t.main}, i.e., to prove the second inequality in \eqref{e.main}. The next result also characterizes the cases of equality.

\begin{theorem}\label{t.main_lower}
If $\underline{x} = (x_1,\dots,x_n)$ with $x_1 \ge \cdots \ge x_n \ge 0$, and $1 \le k \le n$, then:
\begin{equation}\label{e.main_lower}
\hsm_{k/n}(\underline{x}) \le \sym_k(\underline{x}) \, .
% \left[ \frac{\delta_{x_1} + \dots + \delta_{x_n}}{n}\right]_{k/n} 
% \le 
% \sym_k(x_1,\dots,x_n) \, .
\end{equation}
Furthermore, equality holds if and only if :
\begin{equation}\label{e.cases_equality}
k = n \quad \text{or} \quad 
x_1 = \cdots = x_n \quad \text{or} \quad 
x_k = 0 \, . 
\end{equation}
\end{theorem}

\begin{proof}
% Since the inequality becomes trivial in the case $k=n$, we can assume that $k<n$.
Consider the nonnegative $n \times n$ matrix:
\begingroup % keep the change local
\setlength\arraycolsep{1pt}
\begin{equation}\label{e.special}
A \coloneqq 
\begin{pmatrix}
x_1     &   \cdots  &   x_1     &   \, 1 \, &  \, \cdots \, &   \, 1 \, \\[8pt]
\vdots  &           &   \vdots  &   \vdots  &               &   \vdots  \\[8pt]
x_n     &   \cdots  &   x_n     &   \, 1 \, &  \, \cdots \, &   \, 1 \, 
\end{pmatrix}
\quad \text{with $n-k$ columns of $1$'s.}
\end{equation}
\endgroup
Note that:
\begin{equation}
\per(A) = k! \, (n-k)! \,  E^{(n)}_k(x_1,\dots,x_n) 
\end{equation}
and so
\begin{align}
\pmean(A) &= \left(\frac{E^{(n)}_k(x_1,\dots,x_n)}{\binom{n}{k}}\right)^{\frac{1}{n}} \\ 
&= \left[\sym_k(\underline{x}) \right]^{k/n} \, . \label{e.special_pmean}
\end{align}
% where $S = \sym_k(x_1,\dots,x_n)$ is the symmetric mean.

Now let's compute the scaling mean of $A$ using formula \eqref{e.sm_inf_formula}.
Assume that $k<n$.
Given a column vector $v = \left( \begin{smallmatrix} v_1 \\ \vdots \\ v_n \end{smallmatrix} \right)$ with positive entries,
we have:
\begin{equation}
\gmean(Av) = \prod_{i=1}^{n} (sx_i+r)^\frac{1}{n} \, , \quad \text{where } s \coloneqq  \sum_{i=1}^{k} v_i \text{ and } r \coloneqq  \sum_{i=k+1}^{n} v_i \, .
\end{equation}
On the other hand, by the inequality of arithmetic and geometric means,
\begin{equation}
\gmean(v) \le \left(\frac{s}{k}\right)^{\frac{k}{n}} \left(\frac{r}{n-k}\right)^{\frac{n-k}{n}} \, ,
\end{equation}
with equality if $v_1 = \cdots = v_k = \frac{s}{k}$, $v_{k+1} = \cdots = v_n = \frac{r}{n-k}$.
So, in order to minimize the quotient $\frac{\gmean(Av)}{\gmean(v)}$, it is sufficient to consider column vectors $v$ satisfying these conditions.
We can also normalize $s$ to $1$, and \eqref{e.sm_inf_formula} becomes:
\begin{align}
\smean(A) 
&= \inf_{r>0} \frac{\prod_{i=1}^{n} (x_i+r)^\frac{1}{n}}{ \left(\frac{1}{k}\right)^{\frac{k}{n}} \left(\frac{r}{n-k}\right)^{\frac{n-k}{n}} } \\
&= \left[\hsm_{k/n}(\underline{x})\right]^{k/n}\,  \label{e.special_smean}
\end{align}
by \eqref{e.hsm_formula}.
This  formula $\smean(A) =  \left[\hsm_{k/n}(\underline{x})\right]^{k/n}$ also holds for $k=n$, taking the form $\smean(A) = (x_1 \cdots x_n)^{\frac{1}{n}}$; this can be checked either by adapting the proof above, or more simply by using the homogeneity and reflexivity properties of the scaling mean (see \cite{BIP}).

In conclusion, the matrix \eqref{e.special} has scaling and permanental means given by formulas \eqref{e.special_smean} and \eqref{e.special_pmean}, respectively, and the fundamental inequality \eqref{e.our_vdW} translates into $\hsm_{k/n}(\underline{x}) \le \sym_k(\underline{x})$, that is, \eqref{t.main_lower}.

Furthermore, equality holds if and only if the matrix $A$ defined by \eqref{e.special} satisfies $\smean(A) = \pmean(A)$, by formulas \eqref{e.special_smean} and \eqref{e.special_pmean}. 
As mentioned before,  $\smean(A) = \pmean(A)$ if and only if $A$ has rank $1$ or permanent $0$ (see \cite[Thrm.~2.17]{BIP}).
Note that $A$ has rank $1$ if and only if $k=1$ or $x_1=\dots=x_n$.
On the other hand, by \eqref{e.special_pmean}, %by the Frobenius-König theorem \cite{Minc}, 
$A$ has permanent $0$ if and only if $\sym_k(\underline{x})=0$, or equivalently $x_k=0$.
So we have proved that equality $\hsm_{k/n}(\underline{x}) = \sym_k(\underline{x})$ is equivalent to condition \eqref{e.cases_equality}.
\end{proof}

We close this section with some comments on related results. 
%\margin{Added this sentence.}

\begin{remark}
In \cite{HS76}, the asymptotics of the integral \eqref{e.Cauchy_complex_var} are determined using the saddle point method (see e.g. %\ \cite[pp. 526--533]{CH} and 
\cite[Section 15.4]{Simon}). However, for this method to work, the saddle must be steep, that is, the second derivative at the saddle must be large in absolute value.  Major \cite[p.~1987]{Major} %\footnote{Note, however, that the setting considered by Major \cite{Major} is different since he allows random variables to be negative. He considers the case in which the random variable takes two values, one positive and the other negative \cite[Theorems 2 an 2']{Major}. This prompts cancellations in the symmetric means.}
discusses this situation: if the second derivative vanishes, then ``a more sophisticated method has to be applied and only weaker results can be obtained in this case. We shall not discuss this question in the present paper''.
On the other hand, in the general situation covered by our \cref{t.main}, the saddle can be flat.
(It must be noted that the setting considered by Major is different, since he allows random variables to be negative.) 
\end{remark}

%\margin{Replaced the footnote about Major by the final sentence in the Remark.}

%{\sf\color{magenta} added the next two remarks and slightly modified the footnote regarding Major} 

\begin{remark}
Given an arbitrary $n \times n$ non-negative matrix $A$, the permanental and scaling means satisfy the following inequalities (see  \cite[Theorem 2.17]{BIP}),
\begin{equation} \label{ine:sm.pm}
\smean(A) \leq \pmean(A) \leq n(n!)^{-1/n} \smean(A).
\end{equation}
The sequence $(n(n!)^{-1/n})$ is increasing and converges to $e$. In general, as $n$ tends to infinity the permanental mean does not necessarily converges to the scaling mean. However, there are some special classes of matrices for which this is indeed the case: for example, in the repetitive situation covered by the \emph{Generalized Friedland limit} \cite[Theorem~2.19]{BIP}. Note that  $\hsm_{k/n}(\underline{x})$ and      $\sym_k(\underline{x})$  correspond to the $n/k-$th power of the scaling and permanental mean of the matrix $A$, respectively. Therefore,  \eqref{e.main} can be regarded as an improvement of \eqref{ine:sm.pm} for this particular class of matrices.
\end{remark}

\begin{remark} 
A natural extension of symmetric means are \emph{Muirhead means}, see \cite[\S~2.18]{HLP}, \cite[\S~V.6]{Bullen} for definition and properties. 
Accordingly, it should be possible to define a family of barycenters extending the HS barycenters, taking over from \cite[\S~5.2]{BIP}. %\margin{Modified a sentence.}
% Muirhead means also enjoy a \emph{repetition limit} property which, therefore, extends the notion of HS barycenter: see \cite[\S~5.2]{BIP}. 
An analogue of inequality \eqref{e.main_lower} holds in this extended setting, again as a consequence of the key inequality \eqref{e.our_vdW} between matrix means. However, we do not know if inequality \eqref{e.main_upper} can be extended in a comparable level of generality. 
\end{remark}

%%%%%%%%%%%%%%%%%%%%%%%%%%%%%%%%%%%%%%%%%%%%%%%%%%%%%%%%%%%%%%%%%%%%%%%%%
\section{Continuity of the HS barycenter}\label{s.continuity}
%%%%%%%%%%%%%%%%%%%%%%%%%%%%%%%%%%%%%%%%%%%%%%%%%%%%%%%%%%%%%%%%%%%%%%%%%

%\margin{Rewrote the introduction.}

In this section we study the continuity of the HS barycenter as a two-variable function, $(\mu,c) \mapsto [\mu]_c$, defined in the space $\cP(\R_\plus) \times [0,1]$. 
The most natural topology on $\cP(\R_\plus)$ is the weak topology (defined below).
The barycenter function is not continuous with respect to this topology, but, on the positive side, it is lower semicontinuous, except in a particular situation. 
In order to obtain better results, we need to focus on subsets of measures satisfying the natural integrability conditions (usually \eqref{e.HS_integrability}, but differently for the extremal parameters $c=0$ and $c=1$), and endow these subsets with stronger topologies that are well adapted to the integrability assumptions.

In a preliminary subsection, we collect some general facts on topologies on spaces of measures. In the remaining subsections we prove several results on continuity of the HS barycenter. 
And all these results will be used in combination to prove our general ergodic theorem in \cref{s.ergodic}.

% In this section we study the continuity of the HS barycenter as a two variable function, $(\mu,c) \mapsto [\mu]_c$, defined in the space $\cP(\R_\plus) \times [0,1]$.  There are several ways of endowing $\cP(\R_\plus)$ with a topology. Perhaps, the most natural one is the weak topology (see  \cref{sec:spaces_measures}). However, the  HS barycenter map does not behave well with respect to this topology.  There is a dense subset of  $\cP(\R_\plus)$ for which  $[\mu]_c=\infty$. Nevertheless, in \cref{t.LSC_weak}, we prove that the map is lower semi-continuous with respect to the weak topology. A better suited space to address the continuity of the HS barycenter map is the $1$-Wasserstein space. This is a a subset of $\cP(\R_\plus)$ in which certain integrability conditions are assumed. The topology is defined by means of the Kantorovich metric which, in turn, depends on the metric considered in $\R_\plus$, see \cref{sec:spaces_measures} for precise definitions. Appropriate  choices of the metric in $\R_\plus$ yield  $1$-Wasserstein spaces well suited to study the continuity of the HS barycenter in the cases $c\in (0,1)$, $c=0$ and $c=1$, respectively. In \cref{p.cont_PHS}, we 
% prove that the map is continuous except at pairs $(\mu, c)$ satisfying certain critical conditions. Finally, in \cref{p.cont_extremal}, we prove the continuity of the map in the extremal cases $c=0$ and $c=1$. This is obtained with respect to stronger topologies that are well adapted to the natural integrability assumptions of these cases.  

\subsection{Convergence of measures} \label{sec:spaces_measures}

If $(X, \mathrm{d})$ is a separable complete metric space, let  $C_\mathrm{b}(X)$ be the set of all continuous bounded real functions on $X$, and let $\cP(X)$ denote the set of all Borel probability measures on $X$.
Recall (see e.g.\ \cite{Partha_book}) that the \emph{weak topology} is a metrizable topology on $\cP(X)$ according to with a sequence $(\mu_n)$ converges to some $\mu$ if and only if  $\int \phi \, d \mu_n \to \int \phi \, d \mu$ for every test function $\phi \in C_\mathrm{b}(X)$; 
we say that $(\mu_n)$ converges \emph{weakly} to $\mu$, and denote this by $\mu_n \rightharpoonup \mu$. 
The space $\cP(X)$ is Polish, % (separable completely metrizable topological space)
and it is compact if and only if $X$ is compact. 
Despite the space $C_\mathrm{b}(X)$ being huge (nonseparable w.r.t.\ its usual topology if $X$ is noncompact), by \cite[Theorem~II.6.6]{Partha_book} we can nevertheless find a \emph{countable} subset $\cC \subseteq C_\mathrm{b}(X)$ % of bounded continuous on $X$ 
such that, for all $(\mu_n)$ and $\mu$ in $\cP(X)$, 
\begin{equation}\label{e.countable_test}
\mu_n \rightharpoonup \mu \quad \Longleftrightarrow \quad \forall \phi \in \cC, \ {\textstyle \int \phi \, d \mu_n \to \int \phi \, d \mu } \, .
\end{equation}

The following result 
deals with sequences of integrals $\int \phi_n \, d\mu_n$ where not only the measures but also the integrands vary, and bears a resemblance to Fatou's Lemma:

\begin{proposition}\label{p.Fatou}
Suppose that $(\mu_n)$ is a sequence in $\cP(X)$ converging weakly to some measure $\mu$,
and that $(\phi_n)$ is a sequence of continuous functions on $X$ converging uniformly on compact subsets to some function $\phi$.
Furthermore, assume that the functions $\phi_n$ are bounded from below by a constant $-C$ independent of $n$.
Then, $\liminf_{n \to \infty} \int \phi_n \, d\mu_n \geq \int \phi \, d\mu$.
\end{proposition}

Note that, as in Fatou's Lemma, the integrals in \cref{p.Fatou} can be infinite.

\begin{proof}
Without loss of generality, assume that $C=0$. Let $\lambda \in \R$ be such that $\lambda < \int \phi \, d\mu$. By the monotone convergence theorem, there exists $m \in \N$ such that $\int \min \{ \phi, m\} \, d\mu > \lambda$. For a function $\psi:X \to \R$, let $\hat{\psi}(x):=\min \{ \psi(x), m\}$. Note that
\begin{align} \label{ine}
 \int \phi_n \, d\mu_n \geq \int \hat{\phi}_n \, d\mu_n = \int \left(\hat{\phi}_n -\hat{\phi} \right)\, d\mu_n + \int \hat{\phi} \, d\mu_n \, .
\end{align}
By Prokhorov's theorem (see e.g.\ \cite[Theorem~6.7]{Partha_book}), the sequence $(\mu_n)$ forms a \emph{tight} set, that is, for every $\epsilon > 0$, there exists a compact set $K \subseteq X$ such that $\mu_n(X \setminus K) \le \epsilon/(2m)$ for all $n$. Since $(\phi_n)$ converges uniformly on compact subsets to  $\phi$, we obtain:
\begin{equation} \label{eq:cero_en_K}
\lim_{n \to \infty} \left| \int_K \left(\hat{\phi}_n -\hat{\phi} \right)\, d\mu_n  \right| \leq 
\lim_{n \to \infty} \sup_{x \in K} \left| \hat{\phi}_n(x) -\hat{\phi}(x) \right| =0 \, .
\end{equation}
We also have:
\begin{equation} \label{eq:cero_K_comple}
\left| \int_{X \setminus K} \left(\hat{\phi}_n -\hat{\phi} \right)\, d\mu_n  \right| \leq \mu_n(X \setminus K) \sup_ {x \in X \setminus K} \left| \hat{\phi}_n(x) -\hat{\phi}(x) \right| < \frac{\epsilon}{2m} 2m= \epsilon.
\end{equation}
Since $(\mu_n)$ converges weakly to $\mu$, we have $\lim_{n \to \infty} \int \hat{\phi} \, d\mu_n =\int \hat{\phi} \, d\mu > \lambda$. Therefore, combining \eqref{ine}, \eqref{eq:cero_en_K} and \eqref{eq:cero_K_comple}, for sufficiently large values of $n$ we obtain $ \int \phi_n \, d\mu_n > \lambda$. The result now follows.
\end{proof}

The next direct consequence is useful.

\begin{corollary}\label{c.conv_int_weak}
Suppose that $(\mu_n)$ is a sequence in $\cP(X)$ converging weakly to some measure $\mu$,
and that $(\phi_n)$ is a sequence of continuous functions on $X$ converging uniformly on compact subsets to some function $\phi$.
Furthermore, assume that the functions $|\phi_n|$ are bounded by a constant $C$ independent of $n$.
Then, $\int \phi_n \, d\mu_n \to \int \phi \, d\mu$.
\end{corollary}

\medskip

We will also need a slightly stronger notion of convergence. 
Let $\cP_1(X) \subseteq \cP(X)$ denote the set of measures $\mu$ with finite first moment, that is, 
%{\color{red}[give it a name? $1$-Wasserstein space?]}, that is, 
\begin{equation}\label{e.1st_moment}
\int \mathrm{d}(x,x_0) \, d\mu(x) < \infty \, , %\quad \text{for some (and thus any) $x_0 \in X$.}
\end{equation}
Here and in what follows, $x_0 \in X$ is a basepoint which we consider as fixed, the particular choice being entirely irrelevant.
We metrize $\cP_1(X)$ with \emph{Kantorovich metric} (see e.g.\ \cite[p.~207]{Villani}): % (also known as \emph{$1$-Wasserstein} metric) \cite[\S~3.2]{Boga_weak}, \cite[\S~7.1]{Villani}:
\begin{equation}\label{e.KR}
W_1 (\mu, \nu) \coloneqq \sup \left\{ \int \psi \, d(\mu-\nu)  \st  \psi \colon X \to [-1,1] \text{ is $1$-Lipschitz} \right\} \, .
\end{equation}
Of course, the Kantorovich metric depends on the original metric $\mathrm{d}$ on $X$; in fact, it ``remembers'' it, since $W_1(\delta_x , \delta_y) = \mathrm{d}(x,y)$.
The metric space $(\cP_1(X), W_1)$ is called \emph{$1$-Wasserstein space}; it is separable and complete. %, and it is compact if and only if $X$ is compact.
Unless $X$ is compact, 
the topology %induced by the Kantorovich metric 
on $\cP_1(X)$ 
is stronger than the weak topology. In fact, we have the following characterizations of convergence: 

% with respect to the Kantorovich metric:

%\footnote{{\color{red}Distracting remark; delete it?} Let us note that when $X = \R$, weak convergence $\mu_n  \rightharpoonup \mu$ corresponds to pointwise convergence $F_n \to F$ of distribution functions at the continuity points of $F$, while convergence $\mu_n \to \mu$ with respect to Kantorovich distance corresponds to convergence $F_n \to F$ in $L^1$ norm: \cite[p.~110]{Boga_weak}. {\color{red} What about \cite[Thrm.~1.7.1]{Boga_weak}?}} 

\begin{theorem}[{\cite[Theorem~7.12]{Villani}}]\label{t.Villani}
For all $(\mu_n)$ and $\mu$ in $\cP_1(X)$, the following statements are equivalent:
\begin{enumerate}
\item\label{i.Villani_W}
$W_1(\mu_n, \mu) \to 0$.
\item\label{i.Villani_family} 
if $\phi \colon X \to \R$ is a continuous function such that $\frac{|\phi|}{1+\mathrm{d}(\mathord{\cdot},x_0)}$ is bounded, % for some (thus any) $x_0 \in X$, 
then $\int \phi \, d\mu_n \to \int \phi \, d\mu$.
\item\label{i.Villani_one_more} 
$\mu_n \rightharpoonup \mu$ and  $\int \mathrm{d}(\mathord{\cdot},x_0) \, d\mu_n \to \int \mathrm{d}(\mathord{\cdot},x_0) \, d\mu$. % for some (thus any) $x_0 \in X$.
\item\label{i.Villani_tightness}
$\mu_n \rightharpoonup \mu$ and % and  for some (thus any) $x_0 \in X$,
the following ``tightness condition'' holds:
\begin{equation}
\lim_{R \to \infty} \limsup_{n \to \infty} 
\int_{X \setminus B_R(x_0)} [1+\mathrm{d}(\mathord{\cdot},x_0)]\, d\mu_n = 0 \, .
\end{equation}
where $B_R(x_0)$ denotes the open ball of center $x_0$ and radius $R$.
\end{enumerate}
\end{theorem}
%If any of the statements of the \lcnamecref{p.Villani} holds, we say that $\mu_n \to \mu$ in the $\cP_1(X)$ topology.

The next \lcnamecref{l.conv_int_Wass} should be compared to \cref{c.conv_int_weak}: % since the measures converge in a stronger sense, we can relax the assumptions on the sequence of functions. 

\begin{lemma}\label{l.conv_int_Wass}
Suppose that $(\mu_n)$ is a sequence in $\cP_1(X)$ converging to some measure~$\mu$,
and that $(\phi_n)$ is a sequence of continuous functions on $X$ converging uniformly on bounded subsets to some function $\phi$.
Furthermore, assume that the functions $\frac{|\phi_n|}{1+\mathrm{d}(\mathord{\cdot},x_0)}$ are bounded by a constant $C$ independent of $n$.
Then, $\int \phi_n \, d\mu_n \to \int \phi \, d\mu$.
% \begin{equation}
% \int \phi_n \, d\mu_n \to \int \phi \, d\mu \, .
% \end{equation}
\end{lemma}

%{\color{red}HW: Compare the lemma with \cite[Exercise~1.7.20, p.~43]{Boga_weak}, which I found later...}

\begin{proof}
Fix $\epsilon>0$.
By part \eqref{i.Villani_tightness} of \cref{t.Villani}, there exists $R>0$ such that, for all sufficiently large $n$,
\begin{equation}
\int_{X \setminus B_R(x_0)} [1+\mathrm{d}(x,x_0)]\, d\mu_n (x)  \le \epsilon \, . 
\end{equation}
% Let $\psi$ be a continuous function such that:
% \begin{equation}
% 0 \le \psi \le 1, \quad 
% \psi \equiv 1 \text{ on } B_R(x_0), \quad \text{and} \quad
% \psi \equiv 0 \text{ outside of } B_{2R}(x_0) \, .
% \end{equation}
Then, we write:
\begin{multline}
\int \phi_n \, d\mu_n - \int \phi \, d\mu 
=
\int (\phi_n-\phi) \, d\mu_n + \int \phi \, d(\mu_n-\mu) \\
=
\underbrace{\int_{X \setminus B_R(x_0)} (\phi_n-\phi) \, d\mu_n}_{\circled{1}} + \underbrace{\int_{B_R(x_0)} (\phi_n-\phi) \, d\mu_n}_{\circled{2}} + \underbrace{\int \phi \, d(\mu_n-\mu)}_{\circled{3}} \, .
% \underbrace{\int (\phi_n-\phi)(1-\psi) \, d\mu_n}_{\circled{1}} + \underbrace{\int (\phi_n-\phi)\psi \, d\mu_n}_{\circled{2}} + \underbrace{\int \phi \, d(\mu_n-\mu)}_{\circled{3}} \, .
\end{multline}
By part \eqref{i.Villani_family} of \cref{t.Villani}, the term $\circled{3}$ tends to $0$ as $n \to \infty$.
By the assumption of uniform convergence  on bounded sets,
$\left| \circled{2} \right| \le \sup_{B_{2R}(x_0)} |\phi_n - \phi|$ tends to $0$
as well.
Finally,
\begin{equation}
\left| \circled{1} \right| \le \int_{X \setminus B_{R}(x_0)} \left(|\phi_n| + |\phi| \right) \, d\mu_n  \le 2 C \int_{X \setminus B_R(x_0)} [1+\mathrm{d}(x,x_0)]\, d\mu_n (x) \le  2C \epsilon \, ,
\end{equation}
for all sufficiently large $n$.
Since $\epsilon>0$ is arbitrary, we conclude that $\int \phi_n \, d\mu_n \to \int \phi \, d\mu$, as claimed. 
\end{proof}

\medskip

Let us now consider the specific case of the metric space $(X,\mathrm{d}) = (\R_\plus, \mathrm{d}_\mathrm{HS})$, where: 
\begin{equation}\label{e.dist_HS}
\mathrm{d}_\mathrm{HS}(x,y) \coloneqq \left| \log(1+x) - \log(1+y) \right| \, .
\end{equation}
Then the finite-first-moment condition \eqref{e.1st_moment} becomes our usual integrability condition \eqref{e.HS_integrability}, so the $1$-Wasserstein space $\cP_1(X)$ becomes $\PHS$.
% So we have defined a topology on $\PHS$, {\color{red} which makes it meaningful to discuss continuity properties of the HS barycenter.}

\subsection{Lower and upper semicontinuity}
%\margin{Changed the title of the subsection, since it was a repetition of the title of the section}

The HS barycenter is definitely not continuous with respect to the weak topology, since the complement of $\PHS$ is dense in $\cP(\R_\plus)$, and the barycenter is $\infty$ there (by \cref{p.HS_infinity}). Nevertheless, lower semicontinuity holds, except in the critical configuration:

\begin{theorem}\label{t.LSC_weak}
For every $(\mu,c) \in \cP(\R_\plus) \times [0,1)$, we have:
\begin{equation}
\liminf_{(\tilde \mu, \tilde c) \to (\mu,c)} [\tilde \mu]_{\tilde c} \ge [\mu]_c \quad \text{unless} \quad \mu(0)=1-c \text{ and } [\mu]_c > 0 \, .
\end{equation}
\end{theorem}
To be explicit, the inequality above means that for every $\lambda < [\mu]_c$, there exists a neighborhood $\cU \subseteq \cP(\R_\plus) \times [0,1)$ of $(\mu,c)$ with respect to the product topology (weak $\times$ standard) such that $[\tilde \mu]_{\tilde c} > \lambda$ for all $(\tilde \mu, \tilde c) \in \cU$.

\begin{proof}
Let us first note that:
\begin{equation}\label{e.US_discont}
\mu(0)=1-c \text{ and } [\mu]_c > 0 \quad \Rightarrow \quad 
\liminf_{(\tilde \mu, \tilde c) \to (\mu,c)} [\tilde \mu]_{\tilde c} < [\mu]_c \, .
\end{equation}
Indeed, if $c_n \coloneqq c +1/n$, then  $(\mu, c_n) \to (\mu, c)$. By \cref{p.critical_supercritical}, we have
that $[\mu]_{c_n}=0$. Thus,
\begin{equation}
 0= \lim_{n \to \infty} [\mu]_{c_n} =\liminf_{(\tilde \mu, \tilde c) \to (\mu,c)} [\tilde \mu]_{\tilde c} < [\mu]_c \, . 
 \end{equation}
 
We now consider the converse implication: given $(\mu,c) \in \cP(\R_\plus) \times [0,1)$ such that $\mu(0) \neq 1-c$ or $[\mu]_c = 0$, we want to show that $\liminf_{(\tilde \mu, \tilde c) \to (\mu,c)} [\tilde \mu]_{\tilde c} \ge [\mu]_c$.
There are some trivial cases:
\begin{itemize}
\item If $[\mu]_c=0$, then the conclusion is obvious.  
\item If $\mu(0)> 1-c$, then $[\mu]_c=0$ by  \cref{p.critical_supercritical}, and again the conclusion is clear.
\end{itemize}
In what follows, we assume that $\mu(0)<1-c$ and $[\mu]_c>0$.
Fix a sequence $(\mu_n,c_n)$ converging to $(\mu,c)$.
We need to prove that:
\begin{equation}\label{e.aim_for_today}
\liminf_{n \to \infty} [\mu_n]_{c_n} \ge [\mu]_c \, .
\end{equation}
We may also assume without loss of generality that $[\mu_n]_{c_n}<\infty$ for each $n$.
We divide the proof in two cases:

\smallskip
\noindent\textbf{Case $c>0$:} 
We can also assume that $c_n >0$ for every $n$.
By \cref{p.HS_infinity}, the hypothesis  $[\mu_n]_{c_n}<\infty$ means that $\mu_n \in \PHS$.
By Portmanteau's Theorem \cite[Theorem 6.1(c)]{Partha_book}, $\limsup_{n \to \infty} \mu_n(0) \leq \mu(0) < 1-c$. Thus, for sufficiently large values of $n$ we have $\mu_n(0) < 1-c_n$.  
In this setting, the HS barycenter $[\mu_n]_{c_n}$ can be computed by \cref{p.subcritical}. 
Recall from \eqref{e.eta} that $\eta(\tilde{\mu},\tilde{c})$ denotes the unique positive solution of the equation 
$\int \frac{x}{\tilde{c}x + (1-\tilde{c}) \eta} \, d\tilde{\mu} = 1$. 
Note that $\eta(\mu,c)$ is well defined even in the case  $[\mu]_{c}= \infty$ (see \cref{r.eta_other_cases}). 
We claim that:
\begin{equation}\label{e.claim}
\lim_{n \to \infty} \eta(\mu_n, c_n) = \eta(\mu,c) \, .
\end{equation}

\begin{proof}[Proof of the claim]
Fix numbers $\alpha_0$, $\alpha_1$ with $0< \alpha_0 < \eta(\mu,c) < \alpha_1$.
Then a monotonicity property shown in the proof of \cref{p.subcritical} gives:
\begin{equation}\label{e.cansado1}
\int \frac{x}{cx + (1-c) \alpha_0} \, d\mu > 1 > \int \frac{x}{cx + (1-c) \alpha_1} \, d\mu \, .
\end{equation}
Note the uniform bounds:
\begin{equation}\label{e.cloudy}
0 \le \frac{x}{c_n x + (1-c_n) \alpha_i} \le \left(\inf_n c_n\right)^{-1}  \, .
\end{equation}
So, using \cref{c.conv_int_weak}, we see that $\int \frac{x}{c_n x + (1-c_n) \alpha_i} \, d\mu_n  \to \int \frac{x}{cx + (1-c) \alpha_i} \, d\mu$.
In particular, for all sufficiently large $n$,
\begin{equation}\label{e.cansado2}
\int \frac{x}{c_n x + (1-c_n) \alpha_0} \, d\mu_n > 1 > \int \frac{x}{c_n x + (1-c_n) \alpha_1} \, d\mu_n \, ,
\end{equation}
and thus $\alpha_0 < \eta(\mu_n,c_n) < \alpha_1$, proving the claim \eqref{e.claim}.
\end{proof}

For simplicity, write $y_n \coloneqq \eta(\mu_n, c_n)$.
By \cref{p.kernel}.\eqref{i.kernel_mono_x}, 
\begin{equation}\label{e.coffee}
K(x, y_n, c_n) \ge K(0, y_n, c_n) = \log y_n + c_n^{-1} \log(1-c_n) \ge - C
\end{equation}
for some finite $C$, since $\sup_n c_n < 1$ and $\inf_n y_n > 0$.
This allows us to apply \cref{p.Fatou} and obtain:
\begin{equation}\label{e.more_coffee}
\liminf_{n \to \infty} \int K(x, y_n, c_n) \, d\mu_n \ge \int K(x, y_\infty, c) \, d\mu \, ,
\end{equation}
where $y_\infty \coloneqq \lim_{n\to \infty} y_n = \eta(\mu,c)$.
Using formula \eqref{e.sub_conclusion}, we obtain \eqref{e.aim_for_today}.
This completes the proof of \cref{t.LSC_weak} in the case $c>0$.

\smallskip
\noindent\textbf{Case $c=0$:}
By \cref{p.Fatou} we obtain, $\liminf \int x \, d\mu_n \ge \liminf \int x \, d\mu$, that is, $\liminf [\mu_n]_0 \ge [\mu]_c$. So we can assume that $c_n >0$ for every $n$, like in the previous case.

In order to prove \eqref{e.aim_for_today} in the case $c=0$, let us fix an arbitrary positive $\lambda < [\mu]_0 =  \int  x  \, d \mu$, and let us show that $[\mu_n]_{c_n} > \lambda$ for every sufficiently large $n$.
By the monotone convergence theorem, there exists $m \in \N$ such that $\int \min(x,m) \, d\mu(x) > \lambda$. Let $\hat{\mu}$ (resp.\ $\hat{\mu}_n$) be the push-forward of the measure $\hat{\mu}$ (resp.\ $\mu_n$) by the map $x \mapsto \min(x,m)$. Then $[\hat{\mu}]_0 > \lambda$ and, by \cref{p.basic}.\eqref{i.basic_mono_mu}, $[\hat{\mu}_n]_{c_n} \le [\mu_n]_{c_n}$.
Furthermore, we have $\hat\mu_n \rightharpoonup \hat\mu$, since for every $f \in C_\mathrm{b}(\R_\plus)$,
\begin{equation}
\int f \, d\hat{\mu}_n = 
\int f(\min(x,m)) \, d\mu_n(x) \to
\int f(\min(x,m)) \, d\mu(x) = 
\int f \, d\hat{\mu} \, .
\end{equation}
So, to simplify the notations, we remove the hats and assume that the measures $\mu_n$, $\mu$ are all supported in the interval $[0,m]$.

The numbers $\eta(\mu_n,c_n)$ and $\eta(\mu,c)$ are well-defined, as in the previous case, and furthermore they belong to the interval $[0,m]$, by \cref{l.internality_eta}.
On the other hand, by \cref{p.Fatou},
\begin{equation}
\liminf_{n \to \infty} \int \frac{x}{c_n x + (1-c_n) \lambda} \, d\mu_n \ge \int \frac{x}{\lambda} \, d\mu > 1 \, .
\end{equation}
It follows that $\eta(\mu_n,c_n) > \lambda$ for all sufficiently large $n$.
We claim that $\eta(\mu_n,c_n) \to \eta(\mu,0)$, as before in \eqref{e.claim}.
The proof is the same, except that the upper bound in \eqref{e.cloudy} becomes infinite and must be replaced by the following estimate:
\begin{equation}
\left.
\begin{array}{l}
0 \le x \le m  \\
\alpha_i \ge \lambda 
\end{array}
\right\}
\quad 
\Rightarrow
\quad
0 \le \frac{x}{c_n x + (1-c_n) \alpha_i} \le  \frac{m}{c_n m + (1-c_n) \alpha_i} \le \frac{m}{\lambda} \, .
\end{equation}
So a repetition of the previous arguments yields \eqref{e.claim}, then \eqref{e.coffee} and \eqref{e.more_coffee}, and finally \eqref{e.aim_for_today}. 
Therefore, \cref{t.LSC_weak} has been proved in both cases $c>0$ and $c=0$.
\end{proof}

Next, let us  investigate the behaviour of the HS barycenter on the product space $\PHS \times [0,1]$, where $\PHS$ is endowed with the topology defined in the end of \cref{sec:spaces_measures}.

\begin{theorem}\label{p.cont_PHS}
For every $(\mu,c) \in \PHS \times [0,1]$, we have:
\begin{alignat}{5}
\label{e.USC}
\limsup_{(\tilde \mu, \tilde c) \to (\mu,c)} [\tilde \mu]_{\tilde c} &\le [\mu]_c &\quad &\text{unless} &\quad c&=0 &\text{ and } & [\mu]_0 &< \infty \, , \\ 
\label{e.LSC}
\liminf_{(\tilde \mu, \tilde c) \to (\mu,c)} [\tilde \mu]_{\tilde c} &\ge [\mu]_c &\quad &\text{unless} &\quad \mu(0)&=1-c &\text{ and } & [\mu]_c &> 0 \, . 
\end{alignat}
\end{theorem}

\begin{proof}[Proof of part \eqref{e.USC} of \cref{p.cont_PHS}]
Let us start by proving the following implication:
\begin{equation}\label{e.US_discont_again}
\mu \in \PHS, \ 
c=0, \text{ and } [\mu]_0 < \infty \quad \Rightarrow \quad 
\limsup_{(\tilde \mu, \tilde c) \to (\mu,c)} [\tilde \mu]_{\tilde c} >  [\mu]_c \, .
\end{equation}
Consider measures $\mu_n \coloneqq \frac{n-1}{n} \, \mu + \frac{1}{n} \, \delta_n$.
Clearly, $\mu_n \rightharpoonup \mu$; moreover,
\begin{align}
\int \log(1+x) \, d(\mu-\mu_n)(x) = \frac{\log(1+n)}{n} \to 0 \, .
\end{align}
So using characterization \eqref{i.Villani_one_more} of \cref{t.Villani}, we conclude that $\mu_n \to \mu$ in the topology of $\PHS$. On the other hand, $[\mu_n]_0 = \frac{n-1}{n} \, [\mu]_0 + 1 \to [\mu]_0 + 1$.
This proves \eqref{e.US_discont_again}.

Next, let us prove the converse implication. So, let us fix $(\mu,c)$ such that $c\neq 0$ or $[\mu]_0 = \infty$, and let us show that if $(\mu_n, c_n)$ is any sequence in $\PHS \times [0,1]$ converging to $(\mu,c)$, then $\limsup_{n \to \infty} [\mu_n]_{c_n} \le [\mu]_c$. 
This is obviously true if $[\mu]_c = \infty$, so let us assume that $[\mu]_c < \infty$.
Then our assumption becomes $c>0$, so by removing finitely many terms from the sequence  $(\mu_n, c_n)$, we may assume that $\inf_n c_n > 0$.
%all $c_n$ are bigger than some $\gamma > 0$.
Fix some finite number $\lambda > [\mu]_c$.
By \cref{def.HS}, there is some $y_0>0$ such that $\int K(x,y_0,c) \, d\mu(x) < \log \lambda$.
The sequence of continuous functions 
$\frac{|K(x,y_0,c_n)|}{1+\log(1+x)}$ is uniformly bounded, 
as a direct calculation shows.
% ; indeed:
% \begin{equation}
% K_{c_n} (x,y_0) 
% \left\{
% \begin{array}{l}
% \le \log \max\big(c_n y_0^{-1}, 1-c_n\big) + \log(1+x) \\
% \ge K_{c_n}(0,y_0) = \log y_0 + c_n^{-1}\log(1-c_n) 
% \end{array}
% \right.
% \end{equation}
Furthermore, $K(\mathord{\cdot},y_0, c_n) \to K(\mathord{\cdot},y_0,c) $ uniformly on compact subsets of $\R_\plus$.
So \cref{l.conv_int_Wass} ensures that $\int K(x,y_0,c_n) \, d\mu_n(x) \to \int  K(x,y_0,c) \, d\mu(x)$.
Now it follows from \cref{def.HS} that $[\mu_n]_{c_n} < \lambda$ for all sufficiently large $n$.
Since $\lambda>[\mu]_c$ is arbitrary, we conclude that $\limsup_{n \to \infty} [\mu_n]_{c_n} \le [\mu]_c$, as we wanted to show. 
\end{proof}

\begin{proof}[Proof of part \eqref{e.LSC} of \cref{p.cont_PHS}]
First, we prove that, for all $(\mu,c) \in \PHS \times [0,1]$, 
\begin{equation}\label{e.LS_discont}
\mu(0)=1-c, \text{ and } [\mu]_c > 0 \quad \Rightarrow \quad 
\liminf_{(\tilde \mu, \tilde c) \to (\mu,c)} [\tilde \mu]_{\tilde c} < [\mu]_c \, .
\end{equation}
Consider measures $\mu_n \coloneqq \frac{n-1}{n} \, \mu + \frac{1}{n} \, \delta_0$.
Clearly, $\mu_n \to \mu$ in the topology of $\PHS$. 
%On the other hand, by part~\eqref{i.critical} of \cref{p.critical_supercritical}, $[\mu]_c = B(c) [\mu^\plus]_1 > 0$, while by part~\eqref{i.supercritical} of the same \lcnamecref{p.critical_supercritical}, $[\mu_n]=0$.
By \cref{p.critical_supercritical}.\eqref{i.supercritical},  $[\mu_n]_c=0$. 
This proves \eqref{e.LS_discont}. For $c \in [0,1)$ the converse is a direct consequence of \cref{t.LSC_weak}, since the topology on $\PHS$ is stronger. If $c=1$ and $[\mu]_1=0$ then  the result is obvious.
If $[\mu]_1>0$ then, as the example above shows, the result does not hold.
\end{proof}

\subsection{Continuity for extremal values of the parameter}

\cref{p.cont_PHS} shows that the HS barycenter map on $\PHS \times [0,1]$ is not continuous at the pairs $(\mu,0)$ (except if $[\mu]_0 = \infty$), nor at the pairs $(\mu,1)$ (except if $[\mu]_1 = 0$). Let us see that continuity can be ``recovered'' if we impose extra integrablity conditions and work with stronger topologies.

If we use the standard distance $\mathrm{d}(x,y) \coloneqq |x-y|$ on the half-line $X = \R_\plus = [0,+\infty)$, then the resulting $1$-Wasserstein space is denoted $\Pa$.
On the other hand, using the distance 
\begin{equation}\label{e.dist_g}
\mathrm{d}_\mathrm{g}(x,y) \coloneqq \left| \log x - \log y \right| 
\quad \text{on the open half-line } X = \R_{\plus\plus} \, , 
\end{equation}
the corresponding $1$-Wasserstein space will be denoted $\Pg$.
We consider the latter space as a subset of $\cP(\R_\plus)$, since any measure $\mu$ on $\R_{\plus\plus}$ can be extended to $\R_\plus$ by setting $\mu(0) = 0$.
The topologies we have just defined on the spaces $\Pa$ and $\Pg$ are stronger than the topologies induced by $\PHS$; in other words, the inclusion maps below are continuous:
\begin{equation}\label{e.inclusions}
\begin{tikzcd}[row sep=tiny, column sep=small]
\Pa \arrow[rd,hook]	\\
& \PHS \arrow[r,hook]	& \cP(\R_\plus) \\ 
\Pg \arrow[ru,hook]
\end{tikzcd}
\end{equation}

Note that the ``arithmetic barycenter'' $[\mathord{\cdot}]_0$ is finite on $\Pa$, while the ``geometric barycenter'' $[\mathord{\cdot}]_1$ is finite and non-zero on $\Pg$. 

\medskip

Finally, let us establish continuity of the HS barycenter for the extremal values of the parameter with respect to these new topologies:

\begin{proposition}\label{p.cont_extremal}
Consider a sequence $(\mu_n,c_n)$ in $\PHS \times [0,1]$.
\begin{enumerate}
\item\label{i.cont_extremal_0} 
If $(\mu_n, c_n) \to (\mu,0)$ in $\Pa \times [0,1)$, then $[\mu_n]_{c_n} \to [\mu]_0$.
\item\label{i.cont_extremal_1}
If $(\mu_n, c_n) \to (\mu,1)$ in $\cP_\mathrm{g}(\R_\plus) \times (0,1]$, then $[\mu_n]_{c_n} \to [\mu]_1$.
\end{enumerate}
\end{proposition}

\begin{proof}[Proof of part \eqref{i.cont_extremal_0} of \cref{p.cont_extremal} ]
Note that if $c_n=0$ for every $n \in \N$, the result is direct from the definition of the topology in $\Pa$ (use characterization \eqref{i.Villani_one_more} of \cref{t.Villani}). We assume now that $c_n>0$, for every $n \in \N$. It is a consequence of \cref{t.LSC_weak} that:
\begin{equation} 
\liminf_{(\mu_n, c_n) \to (\mu,0)} [\mu_n]_{c_n} \ge [\mu]_0\, .
 \end{equation}
The same proof as that of part \eqref{e.USC} of \cref{p.cont_PHS} can be used to prove,
\begin{equation} 
\limsup_{(\mu_n, c_n) \to (\mu,0)} [\mu_n]_{c_n} \le [\mu]_0\, .
 \end{equation} 
Indeed, in the topology  of  $\Pa$ the HS kernels $K(x,y,c_n)$  satisfy the assumptions of \cref{l.conv_int_Wass}. For this, it suffices to notice that, for any fixed value $y_0>0$, 
the sequence of continuous functions $\frac{|K(x,y_0,c_n)|}{1+x}$ is uniformly bounded, 
and that $K(x,y_0,c_n) \to K(x,y_0,0)$ uniformly on compact subsets of $\R_\plus$.
%we have $\lim_{n \to \infty} K(x,y_0,c_n) =xy_0^{-1} -1 + \log y_0$. Therefore, $|K(x,y_0,c_n)|/(1-x)$ is bounded by a constant independent of $n$.
\end{proof}

\begin{proof}[Proof of part \eqref{i.cont_extremal_1} of \cref{p.cont_extremal} ]
In the case that $c_n=1$ for every $n \in \N$, the result is direct from the topology in $\Pg$ (use characterization \eqref{i.Villani_one_more} of \cref{t.Villani}). We assume that $c_n <1$, for every $n$.
It is a consequence of  \eqref{e.USC} of \cref{p.cont_PHS} that:
\begin{equation} \label{ext1_sup}
\limsup_{(\mu_n, c_n) \to (\mu,1)} [\mu_n]_{c_n} \le [\mu]_1\, .
 \end{equation}
 Recall that the HS barycenter is decreasing in the variable $c$; see \cref{p.basic}.\eqref{i.basic_mono_c}. In particular, $[\mu_n]_{c_n} \geq [\mu_n]_1$, for every $n \in \N$. Noticing that $\log x$ is a test function for the convergence in the topology of $\Pg$, we obtain:
\begin{multline}\label{ext1_inf}
\liminf_{n \to \infty} [\mu_n]_{c_n} \geq \liminf_{n \to \infty} [\mu_n]_1 \\ 
 = \exp \left( \liminf_{n \to \infty} \int \log x \, d \mu_n \right)
= \exp \int \log x  \, d \mu = [\mu]_1\, .
\end{multline}
% \begin{align}\label{ext1_inf1}
% \liminf_{n \to \infty} [\mu_n]_{c_n} \geq \liminf_{n \to \infty} [\mu_n]_1 &\\ 
% \label{ext1_inf2} = \exp \left( \liminf_{n \to \infty} \int \log x \, d \mu_n(x) \right)
% = \exp \int \log x  \, d \mu(x)= [\mu]_1\, .
% \end{align}
The result follows combining \eqref{ext1_sup} and \eqref{ext1_inf}. %, and \eqref{ext1_inf2}.
\end{proof}

The following observation complements part \eqref{i.cont_extremal_1} of \cref{p.cont_extremal}, since it provides a sort of lower semicontinuity property at $c=1$ under a weaker integrability condition:

\begin{lemma}\label{l.final_LSC}
Let $c \in [0,1)$ and let $\mu \in \cP(\R_\plus)$ be such that $\log^-(x) \in L^1(\mu)$.
Then:
\begin{equation}
[\mu]_c \ge \exp \int \left( c^{-1} \log^+(x) - \log^-(x) \right) \, d\mu(x) \, . 
\end{equation}
\end{lemma}

\begin{proof}
By definition, $[\mu]_c \ge \int K(x,1,c) \, d\mu(x)$.
Note that if  $x \ge 1$, then $K(x,1,c)  = c^{-1} \log(cx+1-c) \ge c^{-1} \log x$,
while if $x<1$, then $K(x,1,c) \ge \log x$ by \cref{p.kernel}.\eqref{i.kernel_reflex}.
In any case, $K(x,1,c) \ge c^{-1} \log^+(x) - \log^-(x)$, and the \lcnamecref{l.final_LSC} follows.
\end{proof}

%%%%%%%%%%%%%%%%%%%%%%%%%%%%%%%%%%%%%%%%%%%%%%%%%%%%%%%%%%%%%%%%%%
\section{Ergodic theorems for symmetric and HS means}\label{s.ergodic}
%%%%%%%%%%%%%%%%%%%%%%%%%%%%%%%%%%%%%%%%%%%%%%%%%%%%%%%%%%%%%%%%%%

Symmetric means \eqref{e.def_sym_mean_again} are only defined for lists of nonnegative numbers. %\footnote{In reality, there is a notion of \emph{weighted symmetric means} \cite{Bullen_65}, but it seems to be unrelated to our discussion.}
On the other hand, HS barycenters are defined for probability measures on $\R_\plus$, and are therefore much more flexible objects. In particular, there is an induced concept of HS mean of a  list of nonnegative numbers, which we have already introduced in \eqref{e.def_HS_mean}. 
We can also define the HS mean of a function:

\begin{definition}\label{def.functional_HS}
If $(\Omega, \cF, \P)$ is a probability space, $f \colon \Omega \to \R_\plus $ is a measurable nonnegative function, and $c \in [0,1]$, then the \emph{Hal\'asz--Sz\'ekely mean} (or \emph{HS mean}) with parameter $c$ of the function $f$ with respect to the probability measure $\P$ is:
\begin{equation}\label{e.def_functional_HS}
[f \mid \P]_c \coloneqq [f_* \P]_c \, .
\end{equation}
that is, the HS barycenter with parameter $c$ of the push-forward measure on $\R_\plus$.
In the case $c=1$, we require that $\log f$ is semi-integrable. 
\end{definition}

For arithmetic means, the classical ergodic theorem of Birkhoff states the equality between limit time averages and spatial averages. From the probabilistic viewpoint, Birkhoff's theorem is the strong law of large numbers. 
We prove an ergodic theorem that applies simultaneously to symmetric and HS means, and extends results of \cite{HS76,vanEs}:

\begin{theorem}\label{t.ergodic}
Let $(\Omega, \cF, \P)$ be a probability space, let $T \colon \Omega \to \Omega$ be an ergodic measure-preserving transformation, and let $f \colon \Omega \to \R_\plus$ be a nonnegative measurable function.
Then there exists a measurable set $R \subseteq \Omega$ with $\P(R)=1$ with the following properties. 
For any $\omega \in R$,
for any $c \in [0,1]$ such that 
\begin{equation}\label{e.non_critical}
c \neq 1 - \P(f^{-1}(0))
\end{equation}
and 
\begin{equation}\label{e.integrability_1}
c < 1  \quad \text{unless $\log f$ is semi-integrable,}
\end{equation}
and for any sequence $(c_n)$ in $[0,1]$ tending to $c$, 
we have:
\begin{equation}\label{e.ergodic_HS}
\lim_{n \to \infty} \hsm_{c_n} \big( f(\omega), f(T\omega), \dots, f(T^{n-1} \omega) \big) = [f \mid \P]_c\, ;
\end{equation}
furthermore, for any sequence $(k_n)$ of integers such that $1 \le k_n \le n$ and $k_n/n \to c$, we have:
\begin{equation}\label{e.ergodic_sym}
\lim_{n \to \infty} \sym_{k_n} \big( f(\omega), f(T\omega), \dots, f(T^{n-1} \omega) \big) = [f \mid \P]_c \, .
\end{equation}
\end{theorem}

%\margin{Some changes in the Remark}

\begin{remark} 
Since we allow HS means to take infinity value, we do not need integrability conditions as in \cite{HS76,vanEs}, except for the unavoidable hypothesis \eqref{e.integrability_1}.
In the supercritical case $\P(f^{-1}(0)) > 1-c$, both limits \eqref{e.ergodic_HS} and \eqref{e.ergodic_sym} are almost surely attained in finite time. 
In the critical case $\P(f^{-1}(0)) = 1-c$, strong convergence does not necessarily hold, and the values $\sym_{k_n} \big( f(\omega), \dots, f(T^{n-1} \omega) \big)$ 
may oscillate. However, in the IID setting, 
van Es  proved that the sequence of symmetric means converges \emph{in distribution},  provided that the sequence $(\sqrt{n} (k_n/n -c))$ converges in $[-\infty, \infty]$: see \cite[Theorem A1 (b)]{vanEs}.
\end{remark}

As we will soon see, part \eqref{e.ergodic_HS} of \cref{t.ergodic} is obtained using the results about continuity of the HS barycenter with respect to various topologies proved in \cref{s.continuity}, and then  part \eqref{e.ergodic_sym} follows from the inequalities of \cref{t.main} and \cref{r.Maclaurin}.

To begin the proof, % of \cref{t.ergodic}, 
let us fix $(\Omega, \cF, \P)$, $T$, and $f$ as in the statement, and let  $\mu \coloneqq f_* \P \in \cP(\R_\plus)$ denote the push-forward measure.
Given $\omega \in \Omega$, we consider the  sequence of associated \emph{sample measures}:
\begin{equation}
\mu_n^{\omega} \coloneqq \frac{\delta_{f(\omega)} + \delta_{f(T(\omega))}+ \dots + \delta_{f(T^{n-1}(\omega))}}{n} \, .
\end{equation}
As the next result shows, these sample measures converge almost surely\footnote{\cite[Theorem~II.7.1]{Partha_book} contains a similar result, with essentially the same proof.}: 

\begin{lemma}\label{l.samples}
%Let $(\Omega, \cF, \P)$ be a probability space, let $T \colon \Omega \to \Omega$ be an ergodic measure-preserving transformation, and let $f \colon \Omega \to \R_\plus$ be a nonnegative measurable function.
There exists a measurable set $R \subseteq \Omega$ with $\P(R)=1$ such that for every $\omega \in R$, 
the corresponding sample measures converge weakly to $\mu$:
\begin{equation}
\mu_n^\omega \rightharpoonup \mu \, ;
\end{equation}
furthermore, stronger convergences may hold according to the function $f$:
\begin{enumerate}
\item if $\log(1+f) \in L^1(\P)$, then  $\mu_n^\omega \to \mu$ in the topology of $\PHS$; 
\item if $f \in L^1(\P)$, then          $\mu_n^\omega \to \mu$ in the topology of $\Pa$;
\item if $|\log f| \in L^1(\P)$, then   $\mu_n^\omega \to \mu$ in the topology of $\Pg$. 
\end{enumerate}
\end{lemma}

\begin{proof}
Let $\cC \subset C_\mathrm{b}(\R_\plus)$ be a countable set of bounded continuous functions which is sufficient to test weak convergence, i.e., with property \eqref{e.countable_test}.
For each $\phi \in \cC$, applying Birkhoff's ergodic theorem to the function $\phi \circ f$ we obtain a measurable set $R \subseteq \Omega$ with $\P(R)=1$ such that for all $\omega \in R$,
\begin{equation}
\lim_{n \to \infty} \int \phi \, d\mu^\omega_n = 
\lim_{n \to \infty} \frac{1}{n}\sum_{j=0}^{n-1} \phi (f(T^i \omega)) = 
\int \phi \circ f \, d\P = 
\int \phi \, d\mu \, . 
\end{equation}
Since $\cC$ is countable, we can choose a single measurable set $R$ of full probability that works for all $\phi \in \cC$. Then we obtain $\mu_n^\omega \rightharpoonup \mu$ for all $\omega \in R$.
To obtain stronger convergences, we apply Birkhoff's theorem to the functions $\log(1+f)$, $f$, and $|\log f|$, provided they are integrable, and reduce the set $R$ accordingly.
If, for example, $f$ is integrable, then for all $\omega \in R$ we have:
\begin{equation}
\lim_{n \to \infty} \int x \, d\mu^\omega_n(x) = 
\lim_{n \to \infty} \frac{1}{n} \sum_{j=0}^{n-1} f(T^i \omega) = 
\int f \, d\P = 
\int x \, d\mu(x) \, . 
\end{equation}
Applying part \eqref{i.Villani_one_more} of \cref{t.Villani} with $x_0 = 0$ and $\mathrm{d}(x,x_0) = x$, we conclude that $\mu_n^\omega$ converges to $\mu$ in the topology of $\Pa$.
The assertions about convergence in $\PHS$ and $\Pg$ are proved analogously, using instead the corresponding distances \eqref{e.dist_HS} and \eqref{e.dist_g}. 
\end{proof}

\begin{proof}[Proof of \cref{t.ergodic}]
Let $R$ be the set given by \cref{l.samples}.
By the semi-integrable version of Birkhoff's theorem (see e.g.\ \cite[p.~15]{Krengel}),
we can reduce $R$ if necessary and assume that for all $\omega \in R$, %{\color{magenta} Also $\omega$ in the integral? Yes.}
\begin{equation}\label{e.Birkhoff_log}
\lim_{n \to \infty} \frac{1}{n} \sum_{j=0}^{n-1} \log^\pm(f(T^i \omega)) = 
\int \log^\pm(f(\omega)) \, d\P(\omega) \, . 
\end{equation}
Fix a point $\omega \in R$ and a number $c \in [0,1]$ satisfying conditions \eqref{e.non_critical} and \eqref{e.integrability_1}.
Consider any sequence $(c_n)$ in $[0,1]$ converging to $c$.
Let us prove \eqref{e.ergodic_sym}, or equivalently,
\begin{equation}\label{e.conv_claim}
[\mu_n^\omega]_{c_n} \to [\mu]_c \, .  
\end{equation}
There are several cases to be considered, and in all but the last case  we will use \cref{l.samples}: % {\color{red} (maybe add a bit more detail)}:
\begin{itemize}[itemsep=1ex]

\item First case: $0 \le c < 1$ and $[\mu]_c = \infty$. 
Since $\mu_n^\omega \rightharpoonup \mu$, \eqref{e.conv_claim} is a consequence of \cref{t.LSC_weak}.

\item Second case: $0<c<1$ and $[\mu]_c < \infty$. 
Then $\log(1+f) \in L^1(\P)$. Therefore $\mu_n^\omega \to \mu$ in the topology of $\PHS$, and 
 \cref{p.cont_PHS} implies \eqref{e.conv_claim}.

\item Third case: $c=0$ and $[\mu]_0 < \infty$. 
Then $\log f \in L^1(\P)$, and hence $\mu_n^\omega \to \mu$ in the topology of $\Pa$.
So \eqref{e.conv_claim} follows from \cref{p.cont_extremal}.\eqref{i.cont_extremal_0}.

\item Fourth case: $c=1$ and $[\mu]_1 =0$. 
Then $\log (1+f) \in L^1(\P)$. Thus $\mu_n^\omega \to \mu$ in the topology of $\PHS$, and 
\cref{p.cont_PHS} yields \eqref{e.conv_claim}.

\item Fifth case: $c=1$ and  $0< [\mu]_1 < \infty$. 
Then $|\log f| \in L^1(\P)$. Therefore $\mu_n^\omega \to \mu$ in the topology of $\Pg$, and 
\eqref{e.conv_claim} becomes a consequence of \cref{p.cont_extremal}.\eqref{i.cont_extremal_1}.

\item Sixth case: $c=1$ and $[\mu]_1 = \infty$.
Then $\log^-(f)$ is integrable, but $\log^+(f)$ is not.
If $n$ is large enough then $c_n>0$, so \cref{l.final_LSC} gives: % and \eqref{e.Birkhoff_log},
\begin{align}
\log [\mu_n^\omega]_{c_n} 
&\ge \int \left( c_n^{-1} \log^+(x) - \log^-(x) \right) \, d\mu_n^\omega(x)  
\\
&= \frac{1}{c_n} \cdot \frac{1}{n} \sum_{i=0}^{n-1} \log^+(f(T^i\omega)) 
 -  \frac{1}{n} \sum_{i=0}^{n-1} \log^-(f(T^i\omega)) \, ,
% &= \frac{1}{c_n} \cdot \underbrace{\frac{1}{n} \sum_{i=0}^{n-1} \log^+(f(T^i\omega))}_{\to \infty} 
% -  \underbrace{\frac{1}{n} \sum_{i=0}^{n-1} \log^-(f(T^i\omega))}_{\to \int \log^-(f) \, d\P}   
% \\ &\to \infty \, .
\end{align}
which by \eqref{e.Birkhoff_log} tends to $\infty$.
This proves \eqref{e.conv_claim} in the last case.

\end{itemize}
% We have proved \eqref{e.conv_claim}, which is part \eqref{e.ergodic_HS} of the \lcnamecref{t.ergodic}.

Part~\eqref{e.ergodic_HS} of the \lcnamecref{t.ergodic} is proved, and now let use it to prove part~\eqref{e.ergodic_sym}. 
Consider a sequence $(k_n)$ of integers such that $1 \le k_n \le n$ and $c_n \coloneqq k_n/n$ tends to~$c$.
By \cref{t.main},
\begin{equation}
[\mu_n^\omega]_{c_n} \le 
\sym_{k_n} \big( f(\omega), f(T\omega), \dots, f(T^{n-1} \omega) \big)  \le 
\frac{\binom{n}{k_n}^{-1/k_n}}{B(c_n)}  \, [\mu_n^\omega]_{c_n} \, .
\end{equation}
If $[\mu]_c = \infty$, i.e.\ $[\mu_n^\omega]_{c_n} \to \infty$, then the first inequality forces the symmetric means to tend to $\infty$ as well.
So let us assume that $[\mu]_c$ is finite. 
If $c > 0$ then, by \cref{l.asymp}, the fraction on the RHS converges to $1$ as $n \to \infty$, and therefore we obtain the desired limit \eqref{e.ergodic_sym}.
If $c = 0$, then we appeal to Maclaurin inequality in the form
\begin{equation} 
\sym_{k_n} \big( f(\omega), f(T\omega), \dots, f(T^{n-1} \omega) \big)  \le 
\sym_{1} \big( f(\omega), f(T\omega), \dots, f(T^{n-1} \omega) \big) \, .
\end{equation}
So:
\begin{equation}
[\mu_n^\omega]_{c_n} \le 
\sym_{k_n} \big( f(\omega), f(T\omega), \dots, f(T^{n-1} \omega) \big)  \le 
[\mu_n^\omega]_{0} \, .
\end{equation} 
Since \eqref{e.conv_claim} also holds with $c_n \equiv 0$, we see that all three terms converge together to $[\mu]_c$, thus proving \eqref{e.ergodic_sym} also in the case $c=0$.
\end{proof}

Like Birkhoff's Ergodic Theorem itself, \cref{t.ergodic} should be possible to generalize in numerous directions. For example, part \eqref{e.ergodic_HS} can be easily adapted to flows or semiflows (actions of the group $\R$ or the semigroup $\R_\plus$). One can also consider actions of amenable groups, like \cite{Austin,Navas}. We shall not pursue these matters.
In another direction, let us note that Central Limit Theorems for symmetric means of i.i.d.\ random variables have been proved by Székely~\cite{Szekely_CLT} and van~Es~\cite{vanEs}.

A weaker version of \cref{t.ergodic}, in which the function $f$ is assumed to be bounded away from zero and infinity,  was obtained in \cite[Theorem 5.1]{BIP} as a corollary of a fairly general pointwise ergodic theorem: the \emph{Law of Large Permanents} \cite[Theorem 4.1]{BIP}.  We now briefly discuss a generalization of that result obtained by  Balogh and Nguyen \cite[Theorem 1.6]{Balogh}. Suppose that $T$ is an ergodic measure preserving action of the semigroup $\N^2$ on the space $(X, \mu)$. Given an observable $g: X  \to \R_{\plus\plus}$ and a point $x \in X$, we define an infinite matrix whose  $(i,j)$-entry is $g(T^{(i,j)}x)$. Consider square truncations of this matrix and then take the limit of the corresponding permanental means as the size of the square tends to infinity. It is proved that the limit exists $\mu$-almost everywhere. But not only that, it is also possible to identify the limit. It turns out that it is a \emph{functional scaling mean}. This is a far reaching generalization of the matrix scaling mean \eqref{e.def_sm}: see \cite[Section 3.1]{BIP}. 
% Relevant to the situation studied in this article is the fact that $[f \mid \P]_c^c$ is indeed a functional scaling mean. In particular, it enjoys the properties established in \cite[Section 3.1]{BIP}. 
% The Law of Large Permanents was originally proved for the $\N$-action of a product system \cite[Theorem 4.1]{BIP} and later extended to $\N^2$-actions in \cite{Balogh}.
% It is general enough to imply a result similar to \cref{t.ergodic} for dynamical Muirhead means \cite[Corollary 5.2]{BIP}. Although, again, only for an observable bounded away from zero and infinity.

%\margin{The discussion was too extended, so I \%'ed the final sentences... Ok?}

% {\color{blue} Another rem (not very informative; remove?): when $c_n \equiv c$, the \emph{existence} of the limit in the LHS of \eqref{e.ergodic_HS} follows directly from the subadditive ergodic theorem (as a consequence of log-concavity)...} 

% {\color{blue} What about \textbf{uniquely ergodic} dynamics? When can we guarantee uniform convergence? BTW, see new paper Downarowicz--Weiss 2101.04967}

% {\color{blue} Other notions of ``time'':
% \begin{itemize}
% \item part \eqref{e.ergodic_HS} can be adapted for (semi)flows (actions of the semigroup $\R_\plus$)...
% \item both parts of the theorem can be adapted for actions of discrete ameanable group, as in \cite{Navas} (to check; see also Austin)
% \end{itemize}
% }

%%%%%%%%%%%%%%%%%%%%%%%%%%%%%%%%%%%%%%%%%%%%%%%%%%%%%%%%%%%%%%%%%
\section{Concavity properties of the HS barycenter}\label{s.concavity}
%%%%%%%%%%%%%%%%%%%%%%%%%%%%%%%%%%%%%%%%%%%%%%%%%%%%%%%%%%%%%%%%%

In \cref{s.continuity}, we have studied properties of the HS barycenter that rely on topological structures. 
In this section, we discuss properties that are related to affine (i.e.\ convex) structures.

% {\sf\color{red} ***Let us add some motivation ....optimization?}

\subsection{Basic concavity properties}\label{ss.concavity_basic}

Let us first consider the HS barycenter as a function of the measure.
% To avoid a sterile discussion, we consider only the domain where it is finite.

\begin{proposition}\label{p.conc_mu}
For all $c \in [0,1]$, the function $\mu \in \PHS \mapsto [\mu]_c$ is log-concave.
\end{proposition}

\begin{proof}
By definition, 
\begin{equation}
\log [\mu]_c = \inf_{y>0} \int K(x,y,c) \, d\mu(x) \, .
\end{equation}
For each $c$ and $y$, the function $\mu \mapsto \int K(x,y,c) \, d\mu(x)$ is affine.
Since the infimum of affine functions is concave, we conclude that $\log [\mu]_c$ is concave as a function of~$\mu$. %, proving part~\eqref{i.conc_mu}.
\end{proof}

Next, let us consider the HS barycenter as a function of the parameter.

\begin{proposition}\label{p.conc_c}
For all $\mu \in \PHS \setminus \{\delta_0\}$, the function $c \in [0,1] \mapsto [\mu]_c^c$ 
% (where $0^0 \coloneqq 0$) 
is log-concave.
\end{proposition}

This \lcnamecref{p.conc_c} can be regarded as a version of Newton inequality, which says that for every $\underline{x} = (x_1, \dots, x_n)$, the function
\begin{equation}
k \in \{1,\dots, n\} \mapsto  [\sym_k(\underline{x})]^k
\end{equation}
is log-concave (see \cite[Theorem~51, p.~52]{HLP} or \cite[Theorem 1(1), p.~324]{Bullen}).

\begin{proof}[Proof of \cref{p.conc_c}]
% To prove the second part, 
Note the following trait of the HS kernel:
for all $x\ge 0$ and $y>0$, the function 
\begin{equation}
c \in [0,1] \mapsto c K(x,y,c) = c \log y + \log(cy^{-1} x + 1-c) \in [-\infty, +\infty)
\end{equation}
is concave.
Integrating over $x$ with respect to the given $\mu \in \PHS \setminus \{\delta_0\}$, and then taking infimum over $y$, we conclude that the function $c \in [0,1] \mapsto c \log [\mu]_c$ is concave, as we wanted to show.
\end{proof}

Recall from \cref{def.functional_HS} that the HS mean $[f \mid \P]_c$ of a function $f$ with respect to a probability measure $\P$ is simply the HS barycenter of the push-forward $f_* \P$. Let us now investigate this mean as a function of $f$. The same argument from the proof of \cref{p.conc_c} shows that $f \mapsto [f \mid \P]_c^c$ is log-concave. However, we are able to show more:

\begin{proposition}\label{p.conc_f}
Let $(\Omega, \cF, \P)$ be a probability space.
Let $F$ be the set of nonnegative measurable functions $f$ such that $\log(1+f) \in L^1(\P)$.
For every $c \in (0,1]$, the function $f \in F \mapsto [f \mid \P]_c^c$ is concave.
\end{proposition}

This is a consequence of the fact that $[f \mid \P]_c^c$ is a functional scaling mean (see \cite{BIP}), but for the convenience of the reader we provide a self-contained proof.
We start with the following observation:
if $G$ is the set of positive measurable functions $g$ such that $\log g \in L^1(\P)$, then for all $g \in G$,
\begin{equation}\label{e.gm_trick}
\exp \int \log g \, d\P = \inf_{h \in G} \frac{\int g h \, d\P}{\exp \int \log h \, d\P} \, ,
\end{equation}
with the infimum being attained at $h = 1/g$.
Indeed, this is just a reformulation of the inequality between arithmetic and geometric means.

\begin{proof}[Proof of \cref{p.conc_f}]
Let us first consider the case $c \in (0,1)$.
For every fixed value of $y>0$, the function 
\begin{equation}
g(\omega) \coloneqq \exp( c K(f(\omega) ,y,c) )  =  y^c \big( cy^{-1} f(\omega) + 1 - c \big) 
\end{equation}
belongs to the set $G$ defined above. %; furthermore, it is an affine function of $f(\omega)$.
Using identity \eqref{e.gm_trick}, 
\begin{equation}
\exp \int c K(f(\omega) ,y,c) \, d\P(\omega) = \inf_{h \in G}  \frac{\int y^c ( cy^{-1} f(\omega) + 1 - c )  h(\omega) \, d\P(\omega)}{\exp \int \log h \, d\P} \, .
\end{equation}
Consider this expression as a function of $f$; since it is an infimum of affine functions, it is concave.
Taking the infimum over $y>0$, we conclude that $f \mapsto [f \mid \P]_c^c$ is concave, as claimed. 

The proof of the remaining case $c=1$ is similar, but then we need to extend identity \eqref{e.gm_trick} to functions $g$ in $F$; we leave the details for the reader. 
\end{proof}

\subsection{Finer results}\label{ss.concavity_finer}

For the remaining of this section, we assume that the parameter $c$ is in the range $0<c<1$.
Let us consider  the HS barycenter as a function of the measure again.
We define the \emph{subcritical locus} as the following (convex) subset of $\PHS$:
\begin{equation}\label{e.sc}
\cS_c \coloneqq \big\{\mu \in \PHS \st \mu(0) < 1-c \big\} \, .
\end{equation}
% Recall from \cref{p.critical_supercritical} that the function $[\mathord{\cdot}]_c$ is constant equal to zero on the complement of $\cS_c$.
The function $[\mathord{\cdot}]_c$ restricted to the subcritical locus is well-behaved. It is analytic, in a sense that we will make precise below. 
By \cref{p.conc_mu}, this function is log-concave.
Nevertheless, we will show that it is \emph{not} strictly log-concave.

\smallskip

Let us begin with an abstract definition.

\begin{definition}\label{def.qa}
A real-valued function $f$ defined on a convex subset $C$ of a real vector space is called \emph{quasi-affine} if, for all $x$, $y \in C$,
\begin{equation}
f \left( [x,y] \right) \subseteq [f(x),f(y)] \, ,
\end{equation}
where $[x,y] \coloneqq \{(1-t)x+ty \st 0 \le t \le 1\}$, and the right-hand side is the interval with extremes $f(x)$, $f(y)$, independently of their order.
\end{definition}
The explanation for the terminology is that quasi-affine functions are exactly those that are simultaneouly quasiconcave and quasiconvex (for the latter concepts see e.g.\ \cite[Chapter~3]{ADSZ}).
Note that the level sets of a quasi-affine function are convex.

For $\mu$ in the subcritical locus $\cS_c$, the HS barycenter $[\mu]_c$ can be computed using \cref{p.subcritical}, and this computation relies on finding the solution $\eta = \eta(\mu,c)$ of equation~\eqref{e.eta}. 
Since the integrand in \eqref{e.eta} is monotonic with respect to $\eta$, the function $\eta(\mathord{\cdot},c) \colon \cS_c \to \R_{\plus\plus}$ is quasi-affine. % {\sf\color{red} **Make this a proposition and include analyticity (now part of \cref{l.2nd_derivative}).**}
Concerning the barycenter itself, we have:

\begin{proposition}\label{p.strict}
Let $\mu_0$, $\mu_1 \in \cS_c$.
Then the restriction of the function $[\mathord{\cdot}]_c$ to the segment $[\mu_0,\mu_1]$ is 
log-affine if $\eta(\mu_0 , c) = \eta(\mu_1 , c)$, 
and is strictly log-concave otherwise.
% is strictly log-concave if and only if $\eta(\mu_0 , c) \neq \eta(\mu_1 , c)$. % (where the function $\eta$ was defined in \cref{p.subcritical}).
\end{proposition}

See \cref{f.level_sets}.

\begin{figure}[htb]
\includegraphics[width=.65\textwidth]{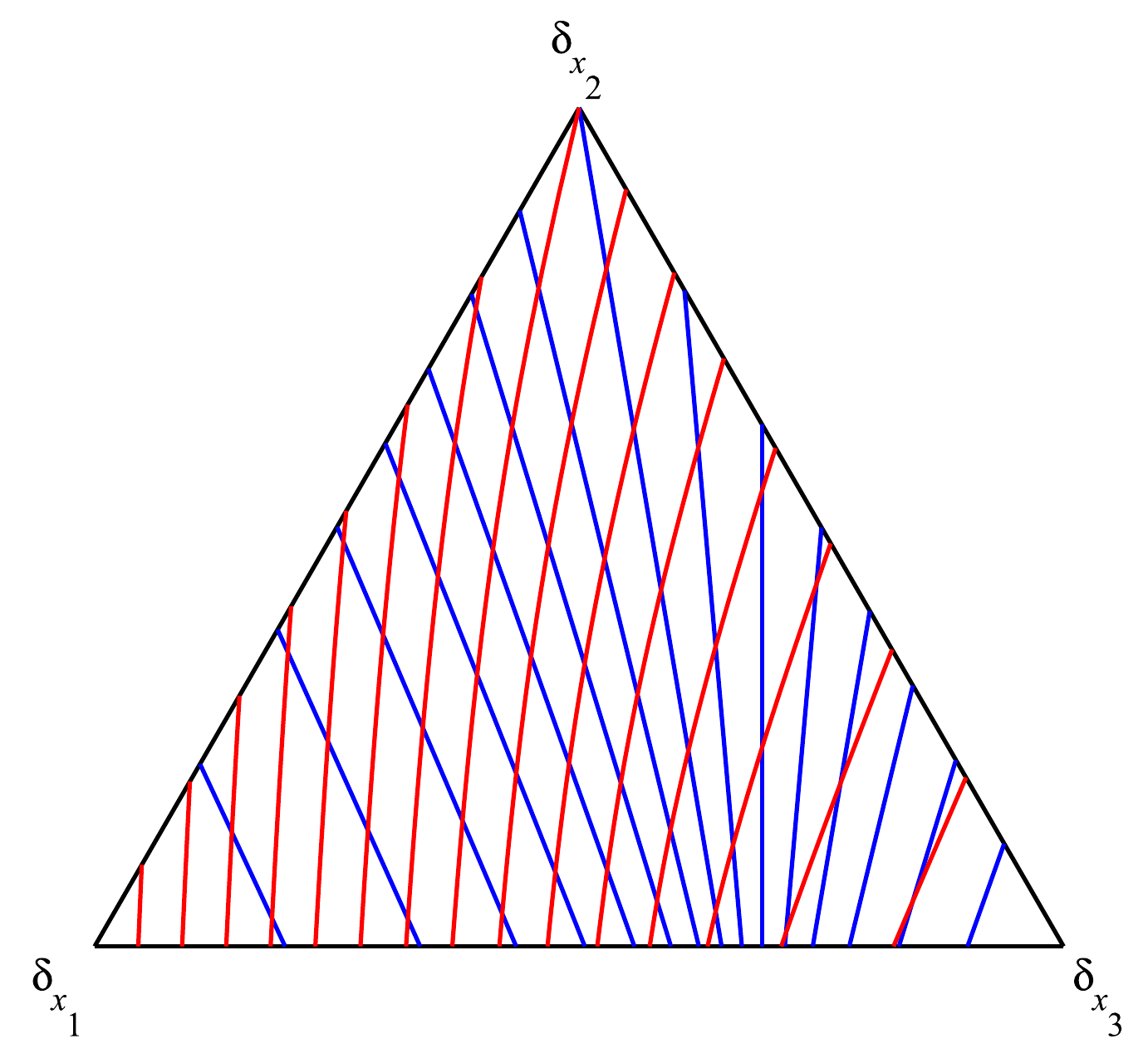}
\caption{The figure shows level curves of the functions $[\mathord{\cdot}]_c$ (in red) and $\eta(\mathord{\cdot},c)$ (in blue) on a $2$-simplex $\Delta\subset\PHS$ whose vertices are distinct delta measures $\delta_{x_1}$, $\delta_{x_2}$, $\delta_{x_3}$. Specifically, we took $c = 2/3$, $x_1 = 1$, $x_2 = 2^4$, $x_3 = 2^8$, and the plotted levels for each function are $2^{0.5}$, $2$, \dots, $2^{7.5}$. The function $\eta(\mathord{\cdot},c)$ is quasi-affine, so the blue level curves are straight segments. The HS barycenter $[\mathord{\cdot}]_c$ is log-affine along each level curve of $\eta(\mathord{\cdot},c)$, and so each blue segment is cut by the red curves into subsegments of equal size (except for the extremes). On the other hand, the HS barycenter $[\mathord{\cdot}]_c$ restricted to any segment $S$ not contained  on a level set of $\eta(\mathord{\cdot},c)$ is strictly log-concave.}\label{f.level_sets}
\end{figure}

\begin{proof}
As observed above, the function $\eta(\mathord{\cdot},c)$ on $\cS_c$ is quasi-affine, and in particular its level sets are convex. As a consequence of \eqref{e.sub_conclusion}, along each level set of $\eta(\mathord{\cdot},c)$, the function $[\mathord{\cdot}]_c$ is log-affine, and so not strictly log-concave there. This proves the first part of the \lcnamecref{p.strict}.

To prove the second part, consider $\mu_0$, $\mu_1 \in \PHS$ such that $\eta(\mu_0 , c) \neq \eta(\mu_1 , c)$, 
and parametrize the segment $[\mu_0,\mu_1]$ by $\mu(t) \coloneqq (1-t)\mu_0+t\mu_1$, $t \in [0,1]$. Then, \cref{l.2nd_derivative} below ensures that the second derivative of the function $t \mapsto \log [\mu]_c$ is nonpositive and vanishes at finitely many points (if any). So the function $t \mapsto [\mu(t)]_c$ is strictly log-concave.
This proves \cref{p.strict}, modulo the \lcnamecref{l.2nd_derivative}.
\end{proof}

\begin{lemma}\label{l.2nd_derivative}
Suppose $I \subset \R$ is an interval and $t \in I \mapsto \mu(t) \in \cS_c$ is an affine mapping.
Write $\mu = \mu(t)$, $\eta = \eta (\mu(t), c)$.
Then $\eta$ and $[\mu]_c$ are analytic functions of~$t$.
Furthermore, letting dot denote derivative with respect to~$t$, the following formula holds:
\begin{equation}\label{e.2nd_derivative}
(\log [\mu]_c)\ddot{\ } = - \frac{(1-c)\dot{\eta}^2}{\eta} \int \frac{x \, d \mu (x) } {\left( cx + (1-c) \eta \right)^2} \, .
\end{equation}
\end{lemma}

The integral is strictly positive (since $\mu \neq \delta_0)$, so formula \eqref{e.2nd_derivative} tells us that, at any point $\mu$ in $\cS_c$, the Hessian of the function $\log [ \mathord{\cdot}]_c$ is negative semidefinite (not a surprise, given \cref{p.conc_mu}), and has the same kernel as the derivative of the function $\eta( \mathord{\cdot},c)$ at the same point.

\begin{proof}[Proof of \cref{l.2nd_derivative}]
Let us omit the parameter $c$ in the formulas, so $K(x,y) = K(x,y,c)$. % and $d\mu=d\mu(x)$.
As in the proof of \cref{p.subcritical}, we consider the following functions:
\begin{equation}
\Delta(x,y) \coloneqq y K_y(x,y) \quad \text{and} \quad
\psi(t,y) \coloneqq \int \Delta(x,y) \, d \mu^{(t)}(x) \, .
\end{equation}
(We temporarily denote $\mu(t)$ by $\mu^{(t)}$.)
Then $\eta(t) =  \eta (\mu^{(t)}, c)$ is defined implicitly by $\psi(t,\eta(t)) = 0$, for all $t \in I$.
The mapping $t \in I \mapsto \mu^{(t)}$ can be extended uniquely to an affine mapping on $\R$ whose values are signed measures. 
Inspecting the proof of \cref{p.subcritical}, we see that $\eta(t)$ is well-defined for all $t$ in an open interval $J \supset I$.

The partial derivative $\Delta_y$ was computed before \eqref{e.Delta_y}, and satisfies the bounds $0 \le \Delta_y \le (2cy)^{-1}$. So we can differentiate under the integral sign and write $\psi_y(t,y) = \int \Delta_y(x,y) \, d \mu^{(t)}(x)$.
This derivative is positive, since $\Delta_y(x,y)>0$ for all $x>0$ and $\mu^{(t)} \neq \delta_0$. 
Therefore, since $\psi$ is an analytic function on the domain $J \times \R_{\plus\plus}$, the inverse function theorem ensures that $\eta$ is analytic on $J$.
In particular, $[\mu^{(t)}]_c =  \exp \int K(x,\eta(t)) \, d \mu^{(t)}(x)$ is analytic as well, as claimed.

% \begin{equation}\label{e.eta_abstract_again}
% \int K_y (x,\eta) \, d \mu = 0 
% \end{equation}

Now we want to prove \eqref{e.2nd_derivative}; in that formula and in the following calculations, we omit the dependence on $t$.
First, we differentiate $L \coloneqq \log [\mu]_c =  \int K(x,\eta) \, d\mu$ with respect to $t$:
\begin{equation}
\dot{L} = \dot{\eta} \int K_y(x,\eta) \, d\mu + \int K(x,\eta) \, d\dot\mu 
\end{equation}
(where $\dot \mu \coloneqq \frac{d}{dt} \mu(t)$ is a signed measure), which by \eqref{e.eta_abstract} simplifies to:
\begin{equation}\label{e.L_dot}
\dot{L} = \int K(x,\eta) \, d\dot\mu  \, .
\end{equation}
Differentiating again, and using that $\ddot{\mu} = 0$ (since $t \mapsto \mu(t)$ is affine), we obtain:
\begin{equation}\label{e.L_dot_dot_1}
\ddot{L} = \dot{\eta} \int K_y(x,\eta) \, d\dot\mu  \, .
\end{equation}
On the other hand, differentiating \eqref{e.eta_abstract},
\begin{equation}
\dot{\eta} \int K_{yy} (x,\eta) \, d \mu + \int K_y(x,\eta) \, d\dot\mu = 0 \, ,
\end{equation}
so \eqref{e.L_dot_dot_1} can be rewritten as:
\begin{equation}\label{e.L_dot_dot_2}
\ddot{L} = -\dot{\eta}^2 \int K_{yy} (x,\eta) \, d \mu \, .
\end{equation}
Let us transform this expression.
Consider the function $\Delta \coloneqq y K_y$, which was introduced before in \eqref{e.Delta}.
Since $\Delta_y = y K_{yy} + K_y$, using \eqref{e.eta_abstract} once again, we obtain $\int \Delta_y(x,\eta) \, d\mu = \eta \int K_{yy} (x,\eta) \, d \mu$.
So equation \eqref{e.L_dot_dot_2} becomes:
\begin{equation}\label{e.L_dot_dot_3}
\ddot{L} = -\frac{\dot{\eta}^2}{\eta} \int \Delta_y (x,\eta) \, d \mu \, .
\end{equation}
Substituting the expression of $\Delta_y$ given by \eqref{e.Delta_y}, we obtain \eqref{e.2nd_derivative}.\footnote{Incidentally, note that it is not true that $K_{yy} \ge 0$ everywhere ($K$ is not a convex function of~$y$), so formula \eqref{e.L_dot_dot_2} by itself is not as useful as the final formulas \eqref{e.L_dot_dot_3} and \eqref{e.2nd_derivative}.}
\end{proof}

% Note that in the situation of \cref{f.level_sets}, the function $[\mathord{\cdot}]_c$ is strictly quasiaffine on the $2$-simplex $\Delta \subset \cS_c$. On the other hand, inside any $3$-simplex whose vertices are distinct delta measures (for example), we can select a level set $T$ of $\eta( \mathord{\cdot},c)$ which is a non-degenerate $2$-simplex, and since the function $[\mathord{\cdot}]_c$ is log-affine on $T$, it is constant on a non-degenerate segment $S \subset T$. This shows that $[\mathord{\cdot}]_c$ is not strictly quasiaffine on $\cS_c$, as claimed before. 

For $c=0$ or $c=1$, the barycenter $[\mathord{\cdot}]_c$ is a quasi-affine function.
On the other hand, an inspection of \cref{f.level_sets} shows that this is not true for $c=2/3$, at least, since the level sets are slightly bent. 
Using \cref{p.strict}, we will formally prove:

\begin{proposition}\label{p.HS_not_qa}
If $c \in (0,1)$, then the function $[\mathord{\cdot}]_c$ is \emph{not} quasi-affine. %on the subcritical locus $\cS_c$.
\end{proposition}

A proof is given in the next section.

%%%%%%%%%%%%%%%%%%%%%%%%%%%%%%%%%%%%%%%%%%%%%%%%%%%%%%%%%%%%
\section{A deviation barycenter related to the HS barycenter}\label{s.DHS}
%%%%%%%%%%%%%%%%%%%%%%%%%%%%%%%%%%%%%%%%%%%%%%%%%%%%%%%%%%%%

There is a large class of means called \emph{deviation means}, which includes the class of quasiarithmetic means.
Let us recall the definition (see \cite{Daroczy,DaroczyPales}). Let $I \subset \R$ be an open interval. A \emph{deviation function} is a function $E \colon I \times I \to \R$ such that for all $x \in I$, the function $y \mapsto E(x, y)$ is continuous, strictly decreasing, and vanishes at $y=x$. % (that is, $E(x,x)=0$) 
Given $n$-tuples $\underline{x}=(x_1,\dots,x_n)$ and $\underline{w}=(w_1,\dots,w_n)$ with $x_i \in I$, $w_i \ge 0$, and $\sum_{i=1}^n w_i = 1$, the \emph{deviation mean} of $\underline{x}$ with weights $\underline{w}$ (with respect to the deviation function $E$) is defined as the unique solution $y \in I$ of the equation: 
\begin{equation}\label{e.def_dev_mean}
\sum_{i=1}^n w_i E(x_i, y) = 0 \, .
\end{equation}
In terms of the probability measure $\mu \coloneqq \sum_{i=1}^n w_i \delta_{x_i}$, this equation can be rewritten as:
\begin{equation}\label{e.def_dev_bar}
\int E(x,y) \, d\mu(x) = 0 \, .
\end{equation}
So it is reasonable to define the \emph{deviation barycenter} of an arbitrary probability $\mu \in \cP(I)$ (with respect to the deviation function $E$) as the solution $y$ of this equation. Of course, existence and uniqueness of such a solution may depend on measurability and integrability conditions, and we will not undertake this investigation here. Nevertheless, let us note that \emph{if $C \subseteq \cP(I)$ is a convex set of probability measures where the deviation barycenter is uniquely defined, then it is a quasi-affine function there.} Indeed, for each $\alpha \in I$, the corresponding upper level set is:
\begin{equation}
\left\{ \mu \in C \st \int E(x,\alpha) \, d\mu(x) \ge 0  \right\}
\end{equation}
and so it is convex; similarly for lower level sets.

\begin{remark}
Let us mention a related concept (see \cite{EM,AL} and references therein).
Let $M$ be a manifold endowed with an affine (e.g.\ Riemannian) connection for which the exponential maps $\exp_y \colon T_y M \to M$ are diffeomorphisms.
Given a probability measure $\mu \in \cP(M)$, 
a solution $y \in M$ of the equation
\begin{equation}\label{e.def_exp_bar}
\int \exp_y^{-1}(x) \, d\mu(x) = 0  
\end{equation}
is called an \emph{exponential barycenter} of $\mu$.
(For criteria of existence and uniqueness, see \cite{AL}.)
The similarity between equations \eqref{e.def_exp_bar} and \eqref{e.def_dev_bar} is evident.
Furthermore, like deviation barycenters, the level sets of exponential barycenters are convex. (Since $M$ has no order structure, it does not make sense to say that the exponential barycenter is quasi-affine.)
\end{remark}

We have mentioned that the HS barycenter with parameter $c \in (0,1)$ is not quasi-affine on the subcritical locus. Therefore HS barycenters are not a deviation barycenters, except for the extremal values of the parameter. Nevertheless, there exists a naturally related parametrized family of deviation barycenters, as we now explain.

Letting $K$ be the HS kernel (see \cref{def.kernel}), we let:
\begin{equation}\label{e.DHS_deviation}
E(x,y,c) \coloneqq K(x,y,c) -\log y = 
\begin{cases}
c^{-1} \log \left( cy^{-1}x + 1 - c \right) & \text{if } c>0, \\ 
y^{-1} x -1 & \text{if } c=0. 
\end{cases}
\end{equation}
For any value of the parameter $c \in [0,1]$, this is a deviation function, 
%Note that for each $c \in [0,1]$, $E(\mathord{\cdot},\mathord{\cdot},c)$ is a deviation function,
provided we restrict it to $x>0$. 
The corresponding deviation barycenter will be called the 
\emph{derived from Hal\'asz--Sz\'ekely barycenter} (or \emph{DHS barycenter}) with parameter $c$. 
More precisely: 

\begin{definition}\label{def.DHS}
Let $c \in [0,1]$ and $\mu \in \cP(\R_\plus)$.
If $c=1$, then we require that the function $\log x$ is semi-integrable with respect to $\mu$.
The \emph{DHS barycenter} with parameter $c$ of the probability measure $\mu$, denoted $\llbracket \mu \rrbracket_c $, is defined as follows:
\begin{enumerate}
\item\label{i.DHS_0} 
if $\mu = \delta_0$, or $c=1$ and $\int \log x \, d\mu(x) = -\infty$, then $\llbracket \mu \rrbracket_c \coloneqq 0$;
\item\label{i.DHS_infinite} 
if $c=0$ and $\int \log x \, d\mu(x) = \infty$, or $c>0$ and  $\int \log (1+x) \, d\mu(x) = \infty$, then $\llbracket \mu \rrbracket_c \coloneqq +\infty$;
\item\label{i.DHS_positive} 
in all other cases, $\llbracket \mu \rrbracket_c$ is defined at the unique positive and finite solution $y$ of the equation $\int E(x,y,c) \, d\mu(x) = 0$.
\end{enumerate}
\end{definition}
Of course, we need to show that the definition makes sense in case \eqref{i.DHS_positive}, i.e., that there exists a unique $y \in \R_{\plus\plus}$ such that  $\int E(x,y,c) \, d\mu(x) = 0$. This is obvious if $c=0$ or $c=1$, so assume that $c \in (0,1)$. Since $\log (1+x) \in L^1(\mu)$, the function $\phi(y) \coloneqq \int E(x,y,c) \, d\mu(x)$ is finite, and (by the dominated convergence theorem) continuous and strictly decreasing. Furthermore, $\phi(y)$ converges to $c^{-1} \log(1-c)<0$ as $y \to +\infty$ and (since $\mu \neq \delta_0)$ to $+\infty$ as $y \to 0^+$. So $\phi$ has a unique zero on $\R_{\plus\plus}$, as we wanted to prove. 

The DHS barycenters have the same basic properties as the HS barycenters (\cref{p.basic}); we leave the verification for the reader.\footnote{For example, monotonicity with respect to $c$ follows simply from the corresponding property of the deviation functions, so we do not need to use the finer comparison criteria from \cite{Daroczy,DaroczyPales}.}
Furthermore, we have  the  following inequality:

\begin{proposition}\label{p.HS_leq_DHS}
$[\mu]_c \le \llbracket \mu \rrbracket_c$.
The inequality is strict unless either $\mu$ is a delta measure, $c=0$, $c=1$, or $\int \log (1+x) \, d\mu(x) = \infty$.
\end{proposition}

\begin{proof}
Let $c \in [0,1]$ and $\mu \in \cP(\R_\plus)$.
If $\llbracket \mu \rrbracket_c$ is either $0$ or $\infty$, then it is clear from \cref{def.DHS} and basic properties of the HS barycenter that $[\mu]_c = \llbracket \mu \rrbracket_c$.
So assume that $\llbracket \mu \rrbracket_c \eqqcolon \xi$ is neither $0$ nor $\infty$. Then it satisfies the equation $\int E(x,\xi,c) \, d\mu(x) = 0$, or equivalently $\int K(x,\xi,c) \, d\mu(x) = \log \xi$. Considering $y = \xi$ in the definition \eqref{e.def_HS}, we obtain $[\mu]_c \le \xi$, as claimed.

Let us investigate the cases of equality. 
It is clear that $\llbracket \mathord{\cdot} \rrbracket_0$ and $\llbracket \mathord{\cdot} \rrbracket_1$ are the arithmetic and geometric barycenters, respectively, and so coincide with the corresponding HS barycenters.
Also, if $\int \log (1+x) \, d\mu(x) = \infty$, then $[\mu]_c = \llbracket \mu \rrbracket_c = \infty$. 
So consider $c \in (0,1)$ and $\mu \in \PHS$ such that $[\mu]_c = \llbracket \mu \rrbracket_c \eqqcolon \xi$.
The infimum in formula \eqref{e.def_HS} is attained at $y = \xi$, and thus (see \cref{p.critical_supercritical}) we are in the subcritical regime $\mu(0)<1-c$. 
Hence equation \eqref{e.eta} holds with $\eta = \xi$. 
Note that the equation can be rewritten as:
\begin{equation}
\int \frac{\eta}{cx+(1-c)\eta} \, d \mu(x) = 1 \, .
\end{equation}
On the other hand, 
\begin{equation}
\int \log \left( \frac{\eta}{cx+(1-c)\eta} \right)  \, d \mu(x) = - c \int E(x,\eta,c) \, d\mu(x) = 0 \, .
\end{equation}
So we have an equality in Jensen's inequality, which is only possible if the integrands are almost everywhere constant, that is, $\mu$ is a delta measure.
\end{proof}

% As an explicit example,
% \begin{multline}
% \nu_1 \coloneqq \delta_1, \quad 	
% \nu_2 \coloneqq \frac{\delta_0 + \delta_4}{2} \quad \Rightarrow \quad
% [\nu_1]_{1/2} = [\nu_2]_{1/2} = 1, \text{ but} \\
% \Big[ \frac{\nu_1 + \nu_2}{2} \Big]_{1/2} = \frac{\sqrt{9+6\sqrt{3}}}{4} \simeq 1.1009  > 1 \, .
% \end{multline}

In some senses, the DHS barycenters are better behaved than HS barycenters. For example, there is no critical phenomena. 

\begin{example}\label{ex.Bernoulli_again}
As in \cref{ex.Bernoulli}, consider the measures $\mu_p \coloneqq (1-p) \delta_0 + p \delta_1$, where $p \in [0,1]$. A calculation gives:
\begin{equation}
\llbracket \mu_p \rrbracket_c = c \left((1-c)^{-\frac{1-p}{p}} - 1 + c\right)^{-1} \, .
\end{equation}
For $c=1/2$, the graphs of the two barycenters are shown in \cref{fig.2graphs}.
\end{example}

\begin{figure}[htb]
\begin{center}
\includegraphics[width=.5\textwidth]{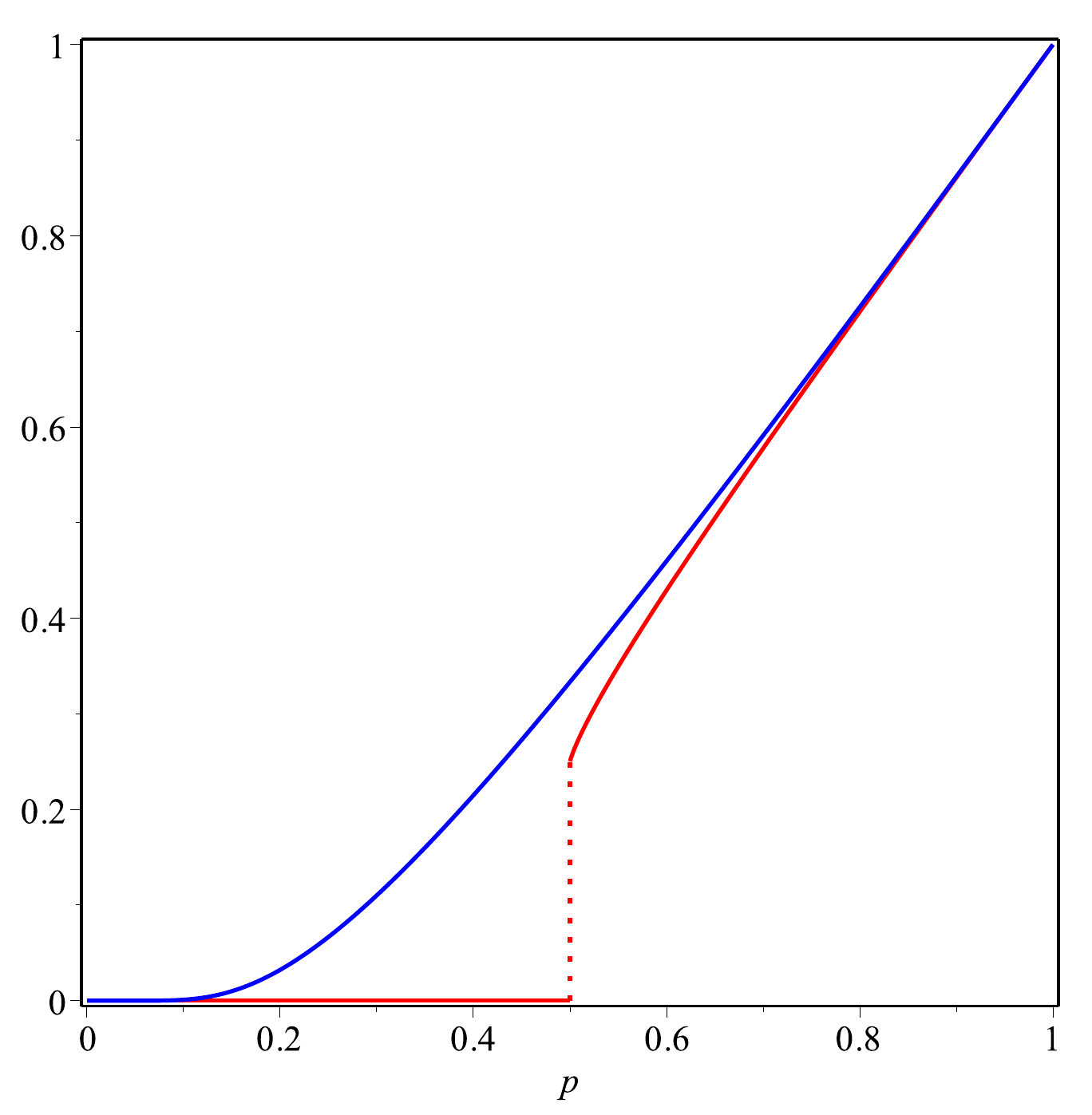}
\caption{Graphs of the functions $p \mapsto [\mu_p]_{1/2}$ (red) and $p \mapsto \llbracket \mu_p \rrbracket_{1/2}$ (blue), where $\mu_p \coloneqq (1-p) \delta_0 + p \delta_1$.}\label{fig.2graphs}
\end{center}
\end{figure}

The definition of DHS barycenters is not so arbitrary as it may seem at first sight; indeed, they approximate HS barycenters:

\begin{proposition}\label{p.tangent}
The HS and DHS barycenters are tangent at $\delta_{x_0}$, for any $x_0>0$.
In other words, if $t \in [0,t_1] \mapsto \mu(t) \in \PHS$ is an affine %a differentiable 
path with $\mu(0) = \delta_{x_0}$, then:
\begin{equation}
\left. \frac{d}{dt} [\mu(t)]_c \right|_{t=0}  = \left. \frac{d}{dt}  \llbracket \mu(t) \rrbracket_c \right|_{t=0} \, .
\end{equation}
% In other words, if $\nu \in \PHS$ and $\mu(t) \coloneqq (1-t) \delta_{x_0} + t \mu$ for $t \in [0,1]$, then:
% \begin{equation}
% [\mu(t)]_c  = \llbracket \mu(t) \rrbracket_c + o(t) \, .
% \end{equation}
\end{proposition}

\begin{proof}
It is sufficient to consider $c \in (0,1)$.
Let us use the notations from the proof of \cref{l.2nd_derivative}.
We evaluate \eqref{e.L_dot} at $t=0$, using $\eta(0) = x_0$, thus obtaining:
\begin{equation}\label{e.L_dot_delta}
\dot{L}(0) = \int K(x,x_0) \, d\dot{\mu} \, .
\end{equation}
Next, consider $\xi = \xi(t) \coloneqq \llbracket \mu(t) \rrbracket_c$.
By definition, $\int K(x,\xi) \, d\mu = \log \xi$.
Differentiating this equation,
\begin{equation}
\dot{\xi} \int K_y (x,\xi) \, d\mu + \int K (x,\xi) \, d\dot\mu = \frac{\dot \xi}{\xi} \, .
\end{equation}
Evaluating at $t=0$ and $\xi(0)=x_0$, we obtain:
\begin{equation}
\dot\xi(0) K_y(x_0,x_0) + \int K(x,x_0) d\dot\mu = \frac{\dot \xi(0)}{x_0} \, .
\end{equation}
But a calculation shows that $K_y(x_0,x_0) = 0$ (this can also be seen as a consequence of part~\eqref{i.kernel_reflex} of \cref{p.kernel}), so we obtain:
\begin{equation}
\dot\xi(0) =  x_0 \int K(x,x_0) \, d\dot{\mu} = x_0 \dot{L}(0) =  \left. \frac{d}{dt} [\mu(t)]_c \right|_{t=0} \, ,
\end{equation}
as we wanted to prove.
\end{proof}

The approximation between the two barycenters is often surprisingly good, even for measures that are not very close to a delta measure:

\begin{example}
If $\mu$ is Lebesgue measure on $[1,2]$ and $c = 1/2$, then: % (up to a rounding error):
\begin{align}
[\mu]_c                     &\simeq 1.485926 \, , \\
\llbracket \mu \rrbracket_c   &\simeq 1.485960 \, ,
\end{align}
a difference of $0.002$\%. 
\end{example}

To conclude our paper, let us confirm that %differently from DHS barycenters, 
HS barycenters are not quasi-affine, except for the extremal cases $c=0$ and $c=1$:

\begin{proof}[Proof of \cref{p.HS_not_qa}]
Let $c \in (0,1)$.
Choose some $\mu_0$ in subcritical locus \eqref{e.sc} which is not a delta measure.
Let $y_0 \coloneqq [\mu_0]_c \in \R_{\plus\plus}$ and $\mu_1 \coloneqq \delta_{y_0}$.
Then $\eta( \mu_1, c) = y_0$.
We claim that $\eta( \mu_0, c) \neq y_0$. 
Indeed, if $\eta( \mu_0, c) = y_0$, then, by \eqref{e.def_HS}, $\log y_0 = \int K(x,y_0,c) \, d\mu_0(x)$, and so $\llbracket \mu_0 \rrbracket_c = y_0$, which by \cref{p.HS_leq_DHS} implies that $\mu_0$ is a delta measure: contradiction. 
Now \cref{p.strict} guarantees that the function $[\mathord{\cdot}]_c$ is strictly log-concave on the segment $[\mu_0,\mu_1]$; in particular, it is not constant. 
Since the function attains the same value on the extremes of the segment, it cannot be quasi-affine.
\end{proof}

\smallskip

%\margin{Added this.}
\noindent \emph{Acknowledgement.} We thank Juarez Bochi for helping us with computer experiments at a preliminary stage of this project.

%%%%%%%%%%%%%%%%%%%%%%%%%%%%%

\end{document}